\documentclass[10pt,twocolumn,twoside]{IEEEtran}  % Comment this line out if you need a4paper
\usepackage{amsmath,amsfonts,amsthm}
\usepackage{tikz}
\usepackage{booktabs}
\usetikzlibrary{automata,arrows,positioning,calc} 
\usepackage[font=footnotesize]{caption}
\usepackage{cite}
\usepackage{color}

\newtheorem{theorem}{Theorem}
\newtheorem{corollary}{Corollary}
\newtheorem{lemma}{Lemma}
\newtheorem{remark}{Remark}
\newtheorem{example}{Example}
\newtheorem{definition}{Definition}
\newtheorem{proposition}{Proposition}
\newtheorem{assumption}{Assumption}

\newcommand{\sharedroads}{\mathcal{N}}
\newcommand{\typesonroad}{\mathcal{M}}
\newcommand{\roadset}{{[}n{]}} 
\newcommand{\vtypeset}{{[}m{]}}
\newcommand{\prevflow}{{z^*}}

\newcommand{\newflow}{{\tilde{z}}}

\usepackage{cleveref}
\usepackage{subcaption}
\usepackage{amsthm}
\usepackage{enumitem}
\newcommand{\hz}{\hat{z}}
\newcommand{\sz}{{z^{*}}}
\newcommand{\hf}{\hat{f}}
\newcommand{\fs}{{f^{*}}}
\DeclareMathOperator*{\argmin}{arg\,min}

\title{\LARGE \bf
The Role of Differentiation in Tolling of Traffic Networks \\ with Mixed Autonomy
}

\author{Daniel A.~Lazar and Ramtin~Pedarsani
	\thanks{The authors are with the Department of Electrical and Computer Engineering, 
		University of California, Santa Barbara
		{\tt\small \{dlazar, ramtin\}@ucsb.edu}}
}

\begin{document}

\maketitle
\thispagestyle{empty}
\pagestyle{empty}

%%%%%%%%%%%%%%%%%%%%%%%%%%%%%%%%%%%%%%%%%%%%%%%%%%%%%%%%%%%%%%%%%%%%%%%%%%%%%%%%
\begin{abstract}
With autonomous vehicles now sharing roads with human drivers, the era of \emph{mixed autonomy} brings new challenges in dealing with congestion. One cause of congestion is when vehicle users choose their routes selfishly to minimize their personal travel delay rather than a global travel delay, and prior works address this phenomenon using tolling to influence routing choices, but do not address the setting of mixed autonomy. Tolls may be \emph{differentiated}, meaning different users of a road experience different tolls, or they may be \emph{anonymous}; the latter is desirable to allay concerns of fairness and privacy, as well as logistical challenges. In this work we examine the role of differentiation in traffic networks with mixed autonomy. Specifically, we first establish differentiated tolls which completely eliminate inefficiency due to selfish routing. We then show the fundamental limitations of anonymous tolls in our setting, and we provide anonymous tolls with mild performance guarantees. We show that in parallel networks, an infinitesimal differentiation in tolls is enough to guarantee optimality, and finally we establish a lower bound on the inefficiency of variable marginal cost tolling in the mixed autonomy setting.
\end{abstract}

%%%%%%%%%%%%%%%%%%%%%%%%%%%%%%%%%%%%%%%%%%%%%%%%%%%%%%%%%%%%%%%%%%%%%%%%%%%%%%%%
\section{INTRODUCTION}
\label{sct:intro}
Autonomous vehicles have been presented as a way to greatly decrease road congestion in popular culture and academic works \cite{noguchi2017npr,lioris2017platoons}. However, recent work shows that in the realm of \emph{mixed autonomy}, where some vehicles are autonomous and some are human-driven, if vehicle users choose routes to minimize their own travel times, switching some vehicles from human-driven to autonomous can paradoxically \emph{worsen} congestion \cite{mehr2019will}. There is also a very large gap between the total delay under selfish routing as compared to the minimum total delay possible (known as the \emph{Price of Anarchy}) in the mixed-autonomy setting \cite{lazar2020routing}.

When drivers consist of a single class of vehicles, well-known tolling schemes can eliminate inefficiency \cite{pigou1932economics,  beckmann1956studies, dafermos1973toll} and the \emph{social cost}, or total delay experienced by all users, can be reduced to the minimum that it would be with all vehicles routed by a social planner. However, in the general mixed autonomy setting, optimal tolls are unknown, absent very restrictive assumptions \cite{mehr2019pricing, lazar2020optimal}. In our work, we provide optimal tolls for a broad class of settings, namely general networks with multiple source-destination pairs and with multiple vehicle types, where road latency is an affine function of vehicle flows. We do so via a simple, and to the best of our knowledge, novel, application of the Variational Inequality. With the tolls we provide, network inefficiency due to selfish routing is eliminated.

The optimal tolls we derive are \emph{differentiated}, meaning that the different vehicle types must experience different tolls. However, it can be difficult to administer these differentiated tolls, due to logistical challenges and concerns about privacy and fairness \cite{iqbal2008designing, rossger2009motivational}. In the case of mixed autonomy, the tolls levied on the users of human-driven vehicles are likely to be greater than those levied against the users of autonomous vehicles, though users of autonomous vehicle may be wealthier. To address this, we investigate \emph{anonymous} tolls, which are identical for all vehicle types on a given road. We provide anonymous tolls with mild performance guarantees and show the strong fundamental limitations on improving performance with anonymous tolls. Because of this fundamental limitation, we provide tolls with only an \emph{infinitesimal} differentiation, which completely eliminate inefficiency due to selfish routing on parallel networks.

The tolls we derive are \emph{fixed tolls} which require global information to compute, including calculating an optimal routing for the given flow demand. Some works have shown that \emph{variable} tolls are needed for tolls to be robust to demand fluctuations or mischaracterizations of road latency functions; a popular toll of this form is variable marginal cost tolling, which guarantee optimality for a wide range of dynamic routing games with a single vehicle type \cite{brown2017studies,sandholm2002evolutionary}. We establish a Price of Anarchy lower bound for networks with variable marginal cost tolls applied and show that this toll structure can result in a large inefficiency in mixed autonomy.

To summarize, our contributions are as follows:
\begin{itemize}
	\item We provide optimal differentiated tolls for general networks with mixed autonomy, with optimality provided via a direct and clear proof,
	\item we extend previous Price of Anarchy results to include networks with more than two vehicle types and affine cost functions,
	\item we provide anonymous tolls with an upper bound on the induced inefficiency and establish lower bounds on the worst-case induced inefficiency with anonymous tolls,
	\item we provide tolls with only an infinitesimal differentiation, which are optimal on parallel networks, and
	\item we provide Price of Anarchy lower bounds for variable marginal cost tolls.
\end{itemize}

\noindent \textbf{Previous Work. } Classic works in congestion games determine optimal tolls for networks with a single vehicle type as far back as Pigou \cite{pigou1932economics}, and canonized in Beckmann \emph{et al.} \cite{beckmann1956studies}. Dafermos generalizes this to the multiclass setting, in which different vehicle types affect congestion differently \cite{dafermos1973toll}. However, this work requires the social cost of flow on each road to be strictly convex, which is generally not the case in the mixed autonomy setting.

In the mixed-autonomy setting with two vehicle types and polynomial latency functions, previous work provides optimal tolls when the asymmetry between the congestion effects of the two vehicle types is constant across all roads in the network \cite{mehr2019pricing}. This work also establishes that when tolls are restricted to be nondifferentiated between vehicle types, tolling may not be able to achieve the socially optimal routing. \cite{lazar2020optimal} drops the assumption of constant asymmetry between vehicle types and provides differentiated tolls for settings with more than two vehicle types, but restricts the class of networks to parallel networks with affine latency functions.

These works are based on the foundation of congestion games, which have long been used to analyze optimal routing \cite{dafermos1969traffic_general}, understand the properties of user (selfish) equilibria \cite{wardrop1900some,depalma1998optimization}, and bound the \emph{Price of Anarchy}, the gap in social cost between optimal routing and the worst-case user equilibrium in a class of games \cite{roughgarden2002bad, correa2008geometric}. Later, this gap was bounded in the setting of mixed autonomy in a work which showed that this gap increases with the ratio between how much the two vehicle types can affect congestion on any road \cite{lazar2020routing}. Similarly, \cite{mehr2019will} describes and bounds the counterintuitively negative effect which can be caused by converting human-driven vehicles to less congesting autonomous vehicles, and \cite{brown2019tragedy} shows a similar effect with autonomous vehicles traveling to park in urban centers.

Some works have studied the robustness of different toll structures to mischaracterization of different parameters \cite{brown2017studies}; the setting of variable marginal cost tolls in evolutionary games has also been studied, which are shown to be optimal in settings with convex latency functions \cite{sandholm2002evolutionary, sandholm2005negative}. This restriction does not describe the setting of mixed autonomy. Our work studies the role of differentiation in tolls for vehicle types which contribute differently to congestion; other works have studied differentiation in tolls for a population with users who are heterogeneous in their value of money as compared to time \cite{brown2016study}.

The effects of mixed autonomy on traffic congestion has been studied in a number of other contexts. This is studied in jointly routing and rebalancing autonomous mobility-on-demand services in the presence of other transportation modes \cite{wollenstein2021routing} as well as autonomous carpooling \cite{amin2021efficient}. Other works study selectively using route recommendations to improve the efficiency of equilibria \cite{zhu2019routing, wu2021value}. Some works consider using autonomous vehicles to ensure the stability of vehicle flow on a road \cite{kreidieh2018dissipating}, or analyze the stability of existing autonomous technology \cite{gunter2020commercially}.

\section{MODEL}
\label{sct:model}

Consider a network described by a graph with $n$ edges, where edges correspond to roads. We consider nonatomic flow, meaning each driver controls an infinitesimally small fraction of vehicle flow. Let $m$ denote the number of vehicle types, where the latency on each edge is a function of the flow of each vehicle type on that edge. We generally use $i$ to index roads and $j$ to index vehicle types. For integer $x$, let ${[}x{]}$ denote $\{1,2, \hdots ,x\}$, and let $\mathbf{1}_n$ denote the $n$-dimensional column vector of ones. Let $\langle v, w \rangle$ denote the inner product between the vectors $v$ and $w$.

Consider inelastic flow demand, with demand described by $\{ s_r, s'_r, \beta_r, j_r \}_{ r \in {[}R{]} }$, where for each \emph{commodity} $r$ in the set of commodities ${[}R{]}$, $s_r$ is that commodity's source node, $s'_r$ is the destination node, $\beta_r$ is the flow demand, and $j_r$ is the vehicle type of the commodity. Thus, each commodity describes the flow demand of a specific vehicle type between a specific source and destination node.

We use $z_i^j$ to describe the flow of vehicle of type $i$ on road $j$. Then, for road $i$, $z_i$ is a column vector describing the flow of all vehicle types on road i: 
$$z_i := \begin{bmatrix}z^1_i & z^2_i & \hdots& z^m_i\end{bmatrix}^T \; .$$ Let $z$ describe the flow over the entire network: $$z =  \begin{bmatrix}z^T_1 & z^T_2 & \hdots& z^T_n\end{bmatrix}^T\; .$$ 

We can alternately describe the flow through a graph in terms of the flow on each route. Let $\mathcal{P}_r$ denote the set of simple paths available to commodity $r$. Let $f^j_p \ge 0$ denote the flow of vehicle type $j$ on path $p$. Then a routing $f$ is feasible if $\sum_{p \in \mathcal{P}_r} f^{j_r}_p = \beta_r$ for commodities $r \in {[}R{]}$. Since there is a one-to-one correspondence between path flows and edge flows, a characterization of feasible routings can be applied to edge routings $z$ as well. Thus, a feasible routing $z$ refers to a routing which satisfies the commodity flow demands and graph constraints. We use $\mathcal{Z}$ to denote the set of feasible routings. We use $\mathcal{I}_p$ to denote the set of edges in path $p$.

As derived in \cite{lazar2019optimal} and \cite{lazar2020optimal} (based on the Bureau of Public Roads delay model in conjunction with a capacity model derived for mixed autonomy \cite{bureau1964manual, lazar2017routing}), we consider multitype mixed autonomous traffic as having an affine latency function, where the scaling effect of each vehicle type on congestion depends on the nominal headway that the vehicle type maintains. Accordingly, each edge $i$ has latency, or cost, function $c_i: \mathbb{R}_{\ge 0}^m \rightarrow \mathbb{R}_{\ge 0}^m$ 
\begin{align*}
	c_i(z_i) = A_iz_i + b_i  \mathbf{1}_n\;
\end{align*}
where $b_i \in \mathbb{R}_{\ge 0}$ denotes the free-flow latency on road $i$ and $A_i \in  \mathbb{R}_{\ge 0}^{m \times m}$ describes the linear increase in latency with flow. We consider all vehicle types to experience the same delay from traversing a road, so we can write $A_i =  \mathbf{1}_n \otimes a_i$, for some $a_i \in \mathbb{R}_{\ge 0}^m$. To ease our later derivations, we define
\begin{align*}
	&A := \begin{bmatrix} A_1 & 0 & \hdots & 0 \\  0 & A_2 & \hdots & 0 \\ \vdots & \vdots & \ddots & \vdots \\ 0 & 0 & \hdots & A_n  \end{bmatrix} \; , \\
	&b := \begin{bmatrix} b_1  \mathbf{1}^T_n & b_2  \mathbf{1}^T_n & \hdots & b_n  \mathbf{1}^T_n  \end{bmatrix}^T \; , \; \text{and} \\
	&c(z) := \begin{bmatrix} c_1(z_1)^T & c_2(z_2)^T & \hdots & c_n(z_n)^T  \end{bmatrix} = Az + b\; ,
\end{align*}
and we use $\mathcal{C}$ to refer to the set of cost functions $c$ which are of this form.

We consider tolling, where vehicle types may experience different tolls. Accordingly, we denote the cost experienced by users of type $j$ on road $i$ under toll $\tau_j$ as $c^{\tau_j}_i \in \mathbb{R}_{\ge 0}^m \rightarrow \mathbb{R}_{\ge 0}^m$, where for some toll $\tau^j_i$,
$$c^{\tau_j}_i(z_i) = c_i(z_i) + \tau^j_i \; .$$

We also define the following quantities:
\begin{align*}	
	&c^\tau_i(z_i) := \begin{bmatrix} c^{\tau_1}_i(z_i)^T & c^{\tau_2}_i(z_i)^T & \hdots & c^{\tau_m}_i(z_i)^T  \end{bmatrix} \; , \\
	&c^\tau(z) := \begin{bmatrix} c^{\tau}_1(z_1)^T & c^{\tau}_2(z_2)^T & \hdots & c^{\tau}_n(z_n)^T  \end{bmatrix} \; , \; \text{and} \\
	& \tau := \begin{bmatrix} \tau^T_1 & \tau^T_2 & \ldots & \tau^T_n  \end{bmatrix}^T \; .
\end{align*}

When discussing the efficiency of a routing, we will judge it by its \emph{social cost}, or aggregate latency, which we denote by $C(z)$, where $C : \in \mathbb{R}_{\ge 0}^{mn} \rightarrow \mathbb{R}_{\ge 0}$:
\begin{equation}\label{eq:objective}
	C(z) := \sum_{i \in {[}n{]}} \langle c_i(z_i),z_i \rangle = \langle c(z), z \rangle \; .
\end{equation}

Note that tolls do not deduct from the social cost, as we consider the tolls to be circulated back into public funds.

We are concerned with the social cost of user equilibria in this setting. This means routings for which no user type would wish to change their routing, where users experience the cost of a route as a sum of the latency and the toll for their user type. Specifically, we consider Wardrop Equilibrium, which can be specified in terms of path routing.
\begin{definition}
	A flow is at \emph{Wardrop Equilibrium} if no user can decrease their cost by switching routes. Mathematically, a flow $f$ is at Wardrop Equilibrium relative to toll $\tau$ if
	\begin{align*}
		\forall r \in {[}R{]}, \; \forall p, p' \in \mathcal{P}_r, &\; f^{j_r}_p > 0  \implies \\ 
		&\sum_{i \in \mathcal{I}_p}c^{\tau_{j_r}}_i(z_i) \le \sum_{i \in \mathcal{I}_{p'}}c^{\tau_{j_r}}_i(z_i) \; .
	\end{align*}
\end{definition}
In words, for each commodity, if a path has positive flow, then no other path available to that commodity can have lower cost, where the cost is relative to the user type of the commodity.

Another useful way of characterizing equilibria is using the \emph{Variational Inequality} \cite{smith1979existence, depalma1998optimization}. Applied to our setting, if feasible routing $z$ is a Wardrop Equilibrium, then for any other feasible routing $z' \in \mathcal{Z}$, $$\langle c^\tau(z), z-z' \rangle \le 0 \; .$$

In contrast with equilibrium routing, we define a socially optimal routing $\sz$ as a routing in the set of feasible routing $\mathcal{Z}$ which minimizes the cost function $C(z)$. We define the set of minimizers of $C(z)$ by $\mathcal{Z}^*$

\section{DIFFERENTIATED TOLLS}
\label{sct:differentiated.tex}
We are now prepared to present optimal tolls which completely eliminate inefficiency in routing in our setting. Specifically, when these differentiated tolls are applied, the social cost of any routing at equilibrium will equal the optimal social cost of any feasible routing.

\begin{theorem}\label{thm:differentiated}
	Let $z^*$ denote a socially optimal routing for a network with social cost $C$. Apply tolls $\tau_i = A_i^T z_i^*$. Any resulting Wardrop Equilibrium $z'$ will have optimal social cost, meaning $$C(z') = C(z^*) \; .$$
\end{theorem}
\begin{proof}
	Let $\hz$ denote a routing at Wardrop Equilibrium, and let $\sz$ denote a socially optimal routing.
	\begin{align*}
		C(\hz) &= \langle c(\hz), \hz \rangle = \langle c^\tau(\hz), \hz \rangle - \tau^T\hz \\
		&\le \langle c^\tau(\hz), \sz \rangle - \tau^T\hz \\
		&= \langle c(\hz),\sz \rangle - \tau^T(\hz - \sz) \\
		&= \langle c(\sz), \sz \rangle + \langle c(\hz)-c(\sz),\sz \rangle - \tau^T(\hz - \sz) \\
		&= C(\sz) + \sz^T(A\hz + b - A\sz - b) - \tau^T(\hz - \sz) \\
		&= C(\sz) + \sz^T A (\hz - \sz) - \tau^T ( \hz - \sz )  \\
		&= C(\sz) + ( \sz^TA - \tau^T )(\hz - \sz) \\
		&= C(\sz)
	\end{align*}
	The inequality on the second line results from the Variational Inequality and the final equality stems from the structure of the tolls. By the definition of optimal routing, $C(\hz) \ge C(\sz)$, so we find that $C(\hz) = C(\sz)$.
\end{proof}

This proof is straightforward, and results from a novel application of Variational Inequality, similar to that of \cite{correa2008geometric}, but applied to the cost function under the application of tolls. Thus we have optimal tolls for multitype mixed autonomy in general networks with affine latency functions. As mentioned in the theorem, the tolls depend on an optimal flow pattern, which must be found to calculate the tolls. 

\begin{remark}
These optimal tolls are the same as the edge tolls in \cite{dafermos1973toll}, though in that work the proof of optimality is limited to when the social cost is a convex function, which in our setting it is not. The optimality of these same tolls has been proven in the mixed autonomy setting when there are two vehicle types \cite{mehr2019pricing}, but it is restricted to the case in which the asymmetry in the congestion effects of the vehicle types is constant across roads, though they prove this for polynomial latency functions. Since they consider only two vehicle types, this assumption means that $a^1_i/a^2_i = k$ for some constant $k$ for all roads $i$ in the set of roads ${[}n{]}$. Our theoretical results greatly expands their results by extending it to multiple types beyond two vehicle types and removing this assumption about the congestion effects of the vehicle types.
\end{remark}

\begin{remark}
Previous results prove the optimality of the path tolls in \cite{dafermos1973toll} in the setting of mutitype mixed autonomy without restrictions on the congestion effects, but limited to parallel networks \cite{lazar2020optimal}. By removing this restrictive assumption, our current results greatly expand the literature on tolling for mixed autonomous vehicle flow to general networks with multiple source-destination pairs.
\end{remark}

\section{ANONYMOUS TOLLS}
\label{sct:anonymous}
\begin{figure*}
	\centering
	\begin{subfigure}[b]{0.5\textwidth}
		\begin{center}
			\centering
			\includegraphics[width=1.0\linewidth]{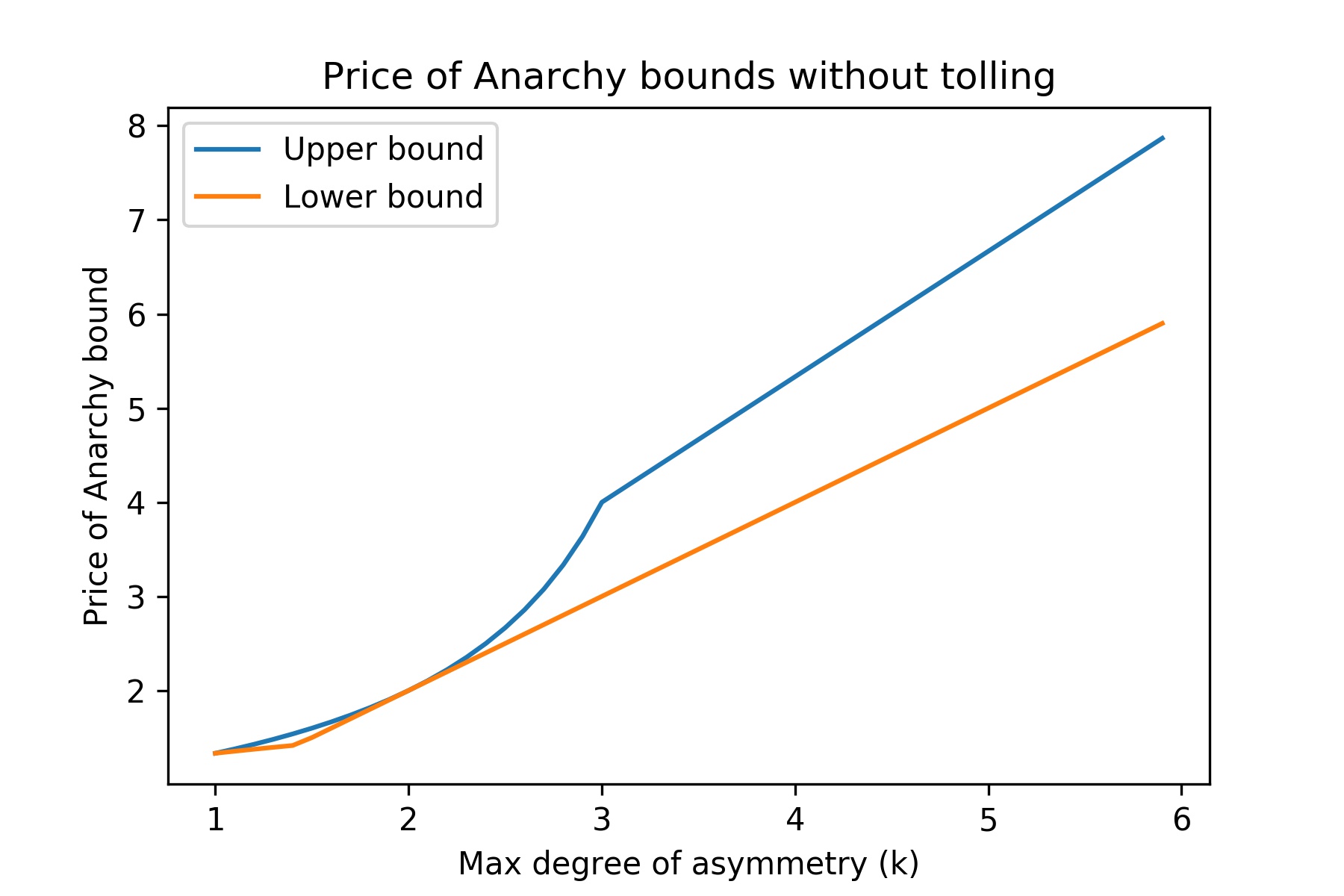}	
			\label{fig:bounds}
		\end{center}
	\end{subfigure}\begin{subfigure}[b]{0.5\textwidth}
		\begin{center}
			\centering
			\includegraphics[width=1.0\linewidth]{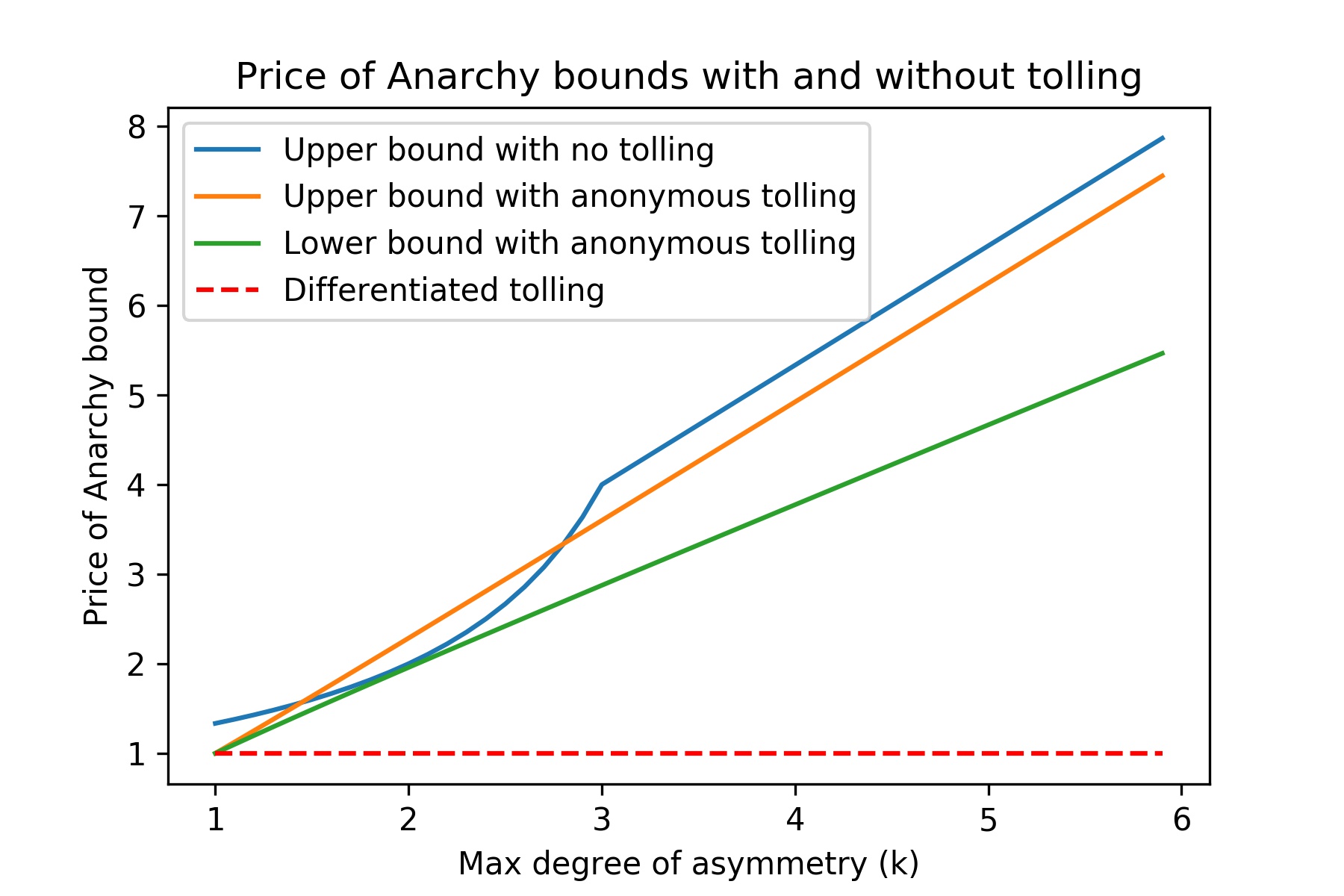}	
			\label{fig:anonymous_tolls}	
		\end{center}
	\end{subfigure}
	\caption{(a) Bounds on the Price of Anarchy in Theorem~\ref{thm:poa}, compared against PoA lower bounds from \cite{lazar2020routing}. (b) Bounds on inefficiency with anonymous tolls, compared with untolled PoA upper bound and tolled PoA lower bound. The PoA with the proposed differentiated tolling is $1$.}
	\label{fig:bounds}
\end{figure*}

As mentioned in the introduction, implementing differentiated tolls may be logistically or politically prohibitive. Accordingly, in this section we present anonymous tolls which improve the bound on worst-case equilibrium behavior. In order to contextualize these results, we first present an extension of previous results which bound worst-case equilibrium behavior in the mixed autonomy setting without tolling. This will contextualize the results on anonymous tolling results to follow.

\subsection{Price of Anarchy Bounds}
The Price of Anarchy (PoA) is a metric which measures, for a class of congestion games and cost functions, the maximum ratio between the social cost under the worst-case equilibrium to the socially optimal cost. This gives a worst-case bound on how much worse-off the social cost can be when users act selfishly. The PoA can be established with or without tolls -- for example, tolls that bring the equilibrium social cost equal to the socially optimal cost, such as the differentiated tolls in the previous section, have a PoA of $1$.

Subsequent examples will show that the class of anonymous tolls we consider cannot achieve a PoA of $1$. In order to understand the PoA reduction we \emph{can} achieve with anonymous tolls, we first introduce an extension of previously established PoA bounds.

Previous results have established the PoA for mixed autonomy when there are only two vehicle types, including for polynomial latency functions \cite{lazar2020routing}. In the current work we are interested in affine latency functions, but for more than two vehicle types. Accordingly, we extend the affine case to our setting.

\begin{definition}
	The \emph{maximum degree of asymmetry}, denoted $k$, is the maximum ratio of congestion effects due to two vehicle types. In our setting,
	$$k:= \max_{i \in {[}n{]}, j,j' \in {[}m{]}} a^j_i/a^{j'}_i \; .$$
\end{definition}

Note that by definition, $k \ge 1$. We now present a theorem which extends the affine case of Theorems~1 and 2 in \cite{lazar2020routing}.
\begin{theorem}\label{thm:poa}
	For any feasible equilibrium routing $\hz$ and optimal routing $\sz$ in an untolled network with maximum degree of asymmetry $k$,
	$$C(\hz) \le \Lambda(k) C(\sz) \; ,$$
	where
	$$ \Lambda(k) = \begin{cases} 4/(4-k) & k \le 3 \\ 4k/3 & k > 3 \end{cases} \; . $$
\end{theorem}
We defer the proof to the appendix.

The bounds above are not proven to be tight in all cases, but \cite{lazar2020routing} provides two examples to bound from below worst-case PoA, one of which reaches a PoA of $k$, and one of which reaches $1+k/(2 \sqrt{k} + 1)$. Figure~\ref{fig:bounds}~(a) compares these lower bounds against the upper bounds from Theorem~\ref{thm:poa}.

\begin{figure*}[h!]
	\centering
	\begin{subfigure}[b]{0.5\textwidth}
		\begin{center}
			\begin{tikzpicture}[->, >=stealth', auto, semithick, node distance=5cm]
			\tikzstyle{every state}=[fill=white,draw=black,thick,text=black,scale=1]
			\node[state]    (0)               {$s$};
			\node[state]    (1)[right of=0]   {$t$};
			\path
			(0) edge[bend left]		node[above]{$c_1(z^1_1,z^2_1) = kz^1_1+ z^2_1$}     (1)
			(0) edge[bend right]		node[below]{$c_2(z^1_2,z^2_2) = z^1_2 + kz^2_2$}     (1);
			\end{tikzpicture}
		\end{center}
	\end{subfigure}\begin{subfigure}[b]{0.5\textwidth}
		\begin{center}
			\begin{tikzpicture}[->, >=stealth', auto, semithick, node distance=5cm]
			\tikzstyle{every state}=[fill=white,draw=black,thick,text=black,scale=1]
			\node[state]    (0)               {$s$};
			\node[state]    (1)[right of=0]   {$t$};
			\path
			(0) edge[bend left]		node[above]{$c_1(z^1_1,z^2_1) = 1$}     (1)
			(0) edge[bend right]	node[below]{$c_2(z^1_2,z^2_2) = \frac{k}{\sqrt{k}+1}z^1_2+\frac{1}{\sqrt{k} + 1}z^2_2$}     (1);
			\end{tikzpicture}
		\end{center}
	\end{subfigure}
	\caption{Examples for anonymous tolling. (a) The proposed anonymous tolling does not improve the PoA from the untolled PoA in this network, with demand of $1$ for each vehicle type. (b) Anonymous tolling improves the PoA in this network, with demand of $1/\sqrt{k}$ of type $1$ and demand of $1$ of type $2$. }
	\label{fig:example_networks}
\end{figure*}
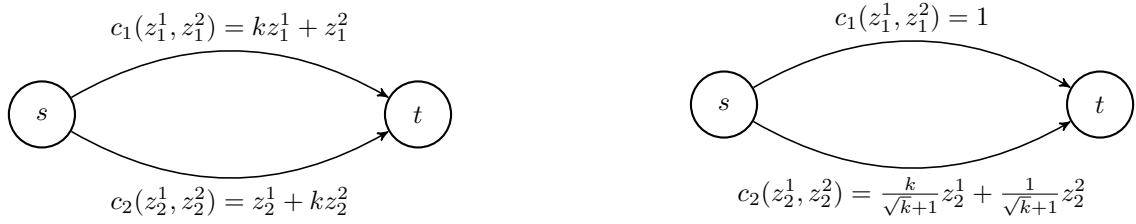

\subsection{Anonymous Tolling}
In this section we establish anonymous tolls which, when applied, improve the previously established PoA bounds; we defer the proof of our theoretical results to the appendix. We will then provide lower bounds on the worst-case PoA with anonymous tolling via examples.

\begin{theorem}\label{thm:anonymous}
	Find an optimal routing $z^*$. Then on each road $i$, levy the following identical toll for all vehicle types: $$\tau^j_i = \min_{j' \in {[}m{]}}a^{j'}_i \sum_{j \in {[}m{]}}z^{*j}_i \; .$$Then, for any Wardrop Equilibrium $\hz$, 
	\begin{equation*}
		C(\hz) \le \frac{4k}{3k+1}kC(\sz) \; .
	\end{equation*}
\end{theorem}

This anonymous toll lowers the upper bound on PoA from that of the class of congestion games without tolling for $k < 1.45$ and $k > 2.8$. However, the PoA bound, while tight for maximum degree of asymmetry $k=1$ and $k=2$, is not tight for all $k$, as shown in Fig.~\ref{fig:bounds}~(a). Therefore, lowering the upper bound is not guaranteed to improve worst-case performance. Nevertheless, the bound with anonymous tolling is lower than without tolling and the subsequent examples show that this tolling can improve the PoA in certain cases.

We next investigate the tightness of this bound. The first example will prove the following proposition, which presents a subclass of anonymous tolls with a specific structure which makes it easier to compute. Specifically, we consider tolling schemes in which the toll on a road is a function of the vehicle flows of each type on that road, the optimal vehicle flows, and the network parameters, but treating these vehicle types interchangeably. We formulate this mathematically in the following proposition.
\begin{proposition}\label{prop:anonymous_tolling}
	Consider the class of tolls in which the anonymous tolls on each road $i$ is a function of the form: $$\tau_i(\{(z^j_i,z^{*j}_i,a^j_i,b_i)\}_{ j \in {[}m{]}}) \; . $$
	The Price of Anarchy with these tolls applied is at least $k$.
\end{proposition}

The lower bound on the PoA in the proposition is proven in the following example.

\begin{example}
Consider the network in Figure~\ref{fig:example_networks}~(a), with flow demand of $1$ for each vehicle type. The socially optimal routing has the flow of type $1$ on the bottom road and type $2$ on the top road, for a social cost of $2$. The worst-case equilibrium has this routing reversed, for a social cost of $2k$. This yields a PoA of $k$.

Due to the symmetry of the optimal routing, the anonymous tolls of Theorem~\ref{thm:anonymous} are equal on the two roads, with a worst-case equilibrium routing again having a social cost of $2k$, yielding a PoA of $k$. Thus, in this network, the anonymous tolls provided do not improve the worst-case performance. Further, due to the symmetry in this example, the two roads have identical inputs to the tolling function for any tolling function of the form proposed in Proposition~\ref{prop:anonymous_tolling}. Thus the two roads must have identical tolls, again yielding a PoA of $k$. This proves the proposition.

This class of tolls in Proposition~\ref{prop:anonymous_tolling} (of which the tolls in Theorem~\ref{thm:anonymous} are a member) has the advantage of depending only on the optimal routing, and not on calculating a specific equilibrium routing, which is often much more computationally difficult to find. However, if the tolls are not restricted to the above form, the social cost can be reduced slightly to $(7k+3)/4 - 1/(k+1)$, as shown in \cite{lazar2020optimal}. The resulting Price of Anarchy in this case is used as our lower bound on PoA for anonymous tolling. We compare the untolled PoA upper bound to the anonymous PoA upper bound and this lower bound in Figure~\ref{fig:bounds}~(b).
\end{example}

We next present an example in which our proposed anonymous tolling \emph{does} improve the Price of Anarchy in a network from what it is without tolling. 

\begin{example}
Consider the network shown in Figure~\ref{fig:example_networks}~(b), with vehicle flow demand of $1/\sqrt{k}$ units of type $1$ and $1$ unit of type $2$. The optimal routing has the flow of type $1$ entirely on the top road and type $2$ on the bottom road, for a cost of $1/\sqrt{k} + 1/(\sqrt{k}+1)$. The worst-case equilibrium has all flow on the bottom road, for a social cost of $1+1/\sqrt{k}$, yielding a PoA which scales with $\sqrt{k}$ \cite{lazar2020routing}.

With the tolling from Theorem~\ref{thm:anonymous}, both vehicle types experience a toll on road $1$ of $0$ and a toll on road $2$ of $1/(\sqrt{k}+1)$. The resulting worst-case equilibrium has all flow of vehicle type $1$ on the lower road and vehicle type $2$ on the upper road, for a social cost of $1 + 1/(\sqrt{k} + 1)$. This is an improvement from the untolled worst-case equilibrium cost of $1+1/\sqrt{k}$.

The worst-case equilibrium social cost could be further improved with a toll on road $2$ of $\tau_2=1/2$, with the resulting worst-case equilibrium involving flow of vehicle type $1$ split between the two roads and all vehicle flow of type $2$ on the top road for a social cost of $1+(3\sqrt{k}-1)/(4k)$. 
\end{example}

\section{INFINITESTIMALLY DIFFERENTIATED TOLLS}
\label{sct:anonymous}
In this section we consider a setting of a parallel network of $n$ roads, where roads have affine latency functions. We present theoretical results with tolls that are almost anonymous and only differentiated with $\epsilon$-differentiation, where $\epsilon$ is arbitrarily small, which completely eliminate inefficiency due to selfish routing. Before presenting our theoretical results, we first introduce some necessary notation.

For a specific routing $z$, we use $\sharedroads^z_j$ to denote the set of roads with positive flow of vehicle type $j$:
\begin{equation*}
	\sharedroads^z_j = \{i : z^j_i > 0 \land i \in \roadset \} \; .
\end{equation*}
Similarly, we use $\typesonroad^z_i$ to denote the set of vehicle types with positive flow on road $i$ for the routing $z$:
\begin{equation*}
	\typesonroad^z_i = \{j : z^j_i > 0 \land j \in \vtypeset \} \; .
\end{equation*}

The theoretical results established in the next section conceptualize vehicle flow on roads in the form of a graph, where for each specific routing $z$, a graph can be constructed. We construct a bipartite graph $G=(U,V,E)$ where one set of nodes is the set of roads ($U = \roadset$) and one set of nodes is the set of vehicle types ($V=\vtypeset$). The set of edges connect vehicle types to roads on which they have positive flow, \emph{i.e.}
\begin{equation}\label{eq:graph}
	E = \{ (i, j) : i \in \roadset \land  j \in \typesonroad^z_i \} \; ,
\end{equation}
or equivalently, $E = \{ (i, j) : z^j_i>0 \land i \in \roadset \land j \in \vtypeset \}$. In other words, for a routing $z$, there is an edge between the nodes denoting road $i$ and vehicle type $j$ if there is positive flow of type $j$ on road $i$. We illustrate this in Figure~\ref{fig:cyclic_acyclic}.

\begin{figure}
	\centering
	\includegraphics[width=\linewidth]{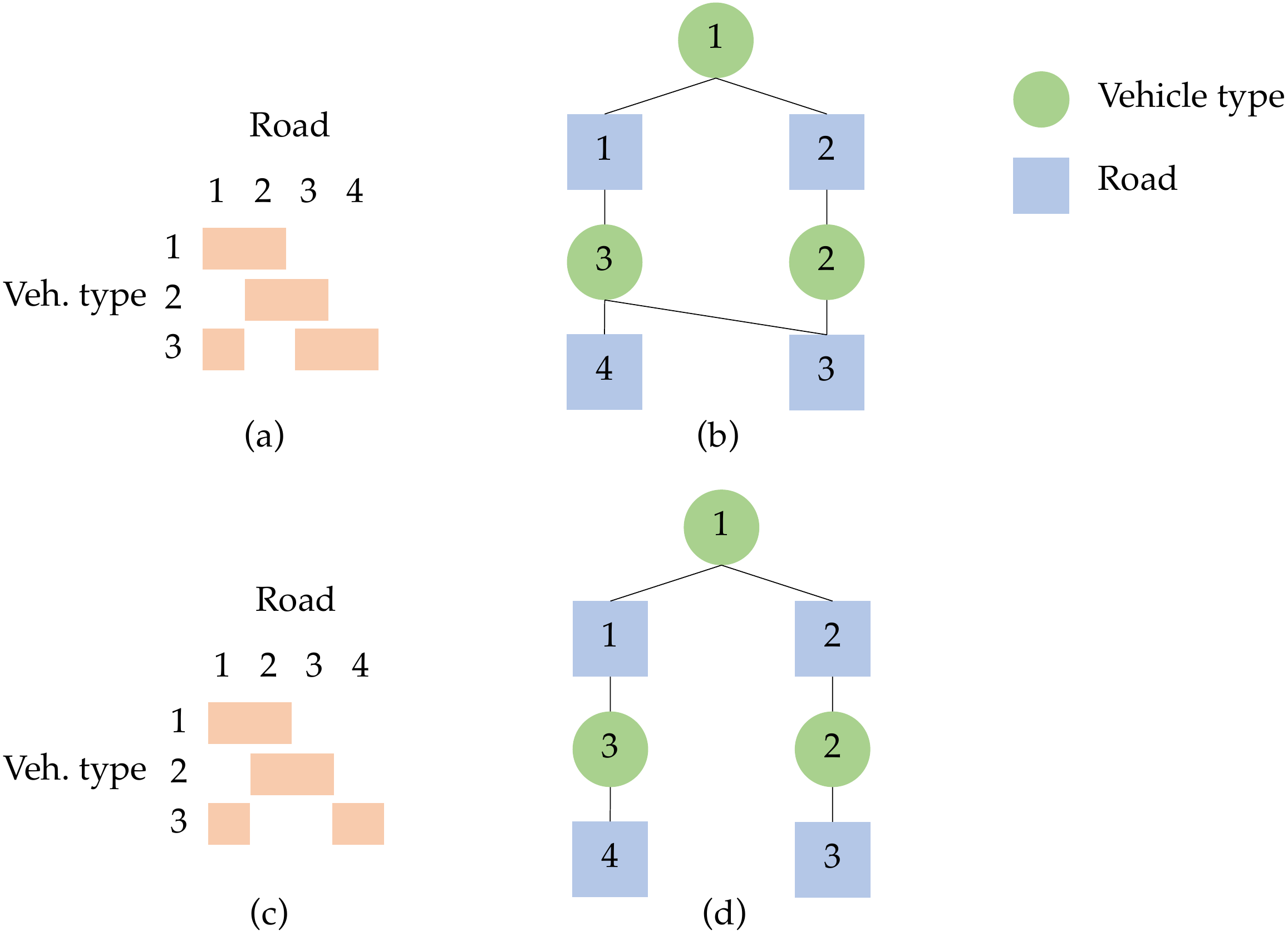}
	\caption{Example routings for a network with four roads and three vehicle types \cite{lazar2020optimal}. (a) vehicle type 1 has positive flow on roads 1 and 2, type 2 has positive flow on roads 2 and 3, and type 3 has positive flow on roads 1, 3, and 4. (b) shows the corresponding bipartite graph. (c) shows a similar routing but type 3 has zero flow on road 3, and (d) shows its corresponding bipartite graph.}
	\label{fig:cyclic_acyclic}
\end{figure}

We make the following assumption regarding the latency functions.
\begin{assumption}\label{asmp:increasing_cost}
	The road latency functions are affine and the latency of each road is strictly increasing with the flow of each vehicle type on that road, \emph{i.e.} $a^j_i>0$ for all $i\in \roadset$, $j \in \vtypeset$.
\end{assumption}

We rephrase the Wardrop Equilibrium condition for the case of parallel networks.
\begin{definition}\label{def:NE}
	A flow $z$ is a \emph{Wardrop Equilibrium} if $z^j_i > 0$ implies $c^j_i(z) \le c^j_{i'}(z)$ for all $i, i' \in \roadset$, $j \in \vtypeset$.
\end{definition}
Since in this section we consider a network of parallel roads, a flow is at Wardrop Equilibrium if no user can decrease their cost by switching roads.

In this section we establish tolls which ensure that the social cost is minimized in any resulting equilibrium. We do this in two theorems: the first establishes properties about a routing in the set of routings which minimizes the social cost, and the second provides optimal tolls which are constructed based on the routing which is proved to exist in the first theorem.

\begin{theorem}[\cite{lazar2020optimal} Theorem~1]\label{thm:routing}
	Consider the setting of multitype congestion games with affine cost functions on parallel networks. There exists a routing in the set of routings minimizing social cost, $z \in \mathcal{Z}^*$, such that $G(z)$ is acyclic, where $G=(U,V,E)$ is constructed as in \eqref{eq:graph}, \emph{i.e.} where nodes are the roads and vehicle types, and edges exist between road $i$ and vehicle type $j$ when $z^j_i>0$.
\end{theorem}

\begin{corollary}\label{cor:mixed}
	For the routing $z \in \mathcal{Z}^*$, such that $G(z)$ is acyclic, provided by Theorem \ref{thm:routing}, no two vehicle types share more than one common road.
\end{corollary}

\begin{theorem}\label{thm:almost_anonymous}
Consider the setting of multitype congestion games with affine cost functions on parallel networks. Consider any optimal routing $\prevflow \in \mathcal{Z}^*$ that has an associated acyclic graph. The existence of such a routing is provided in Theorem~\ref{thm:routing}. Assume the latency on a road is increasing in the flow of each vehicle type on that road. Then levy the following tolls $\tau(\prevflow)$:
	\begin{equation}\label{eq:tolls}
		\tau^j_i(\prevflow_i) = \begin{cases}
		\mu- c_i(\prevflow_i) & \text{if} \; i \in \sharedroads^\prevflow_j \\
		\mu- c_i(\prevflow_i) + \epsilon & \text{if} \; i \in \cup_{j' \in {[}m{]} }\sharedroads^\prevflow_{j'} \setminus \sharedroads^\prevflow_{j} \\
		P & \text{otherwise} .
	\end{cases}
	\end{equation}
	Then for every $\mu$, and every $\epsilon >0$, and sufficiently large $P$, the only Nash equilibrium that exists is the flow $\prevflow$.
\end{theorem}

In words, we find an optimal routing with the acyclic property guaranteed by Thm~\ref{thm:routing}. Denote this routing $z^*$. Then, based on this routing, we assign a toll for each vehicle type on each road. If $z^*$ has no vehicle flow on a specific road, we set a very high toll for all vehicle types, ensuring that in equilibrium, vehicle flow on this road will be zero for all vehicle types. If a vehicle type has positive flow on road $i$ under $z^*$, it is given the toll $\mu - c_i(z_i)$. If it does not have positive flow under $z^*$ but another vehicle type does, it is given the toll $\mu - c_i(z_i) + \epsilon$, where $\epsilon$ is any positive number. This guarantees the optimal routing at equilibrium.

We see that optimal routing can be guaranteed with an infinitesimal difference between the tolls of each vehicle type on a road. We can even consider the setting in which users are given random gifts to help encourage them to follow the route recommendations and use these to ensure optimal routing without having a different toll for the vehicle types. We now prove the theorem and defer the proofs of the supporting lemmas to the appendix.

\begin{proof}

We prove this theorem by breaking down a candidate equilibrium $\newflow$ into cases. We base these cases on whether roads, which previously had flow of certain vehicle types in the optimal routing on which the tolls are based ($\prevflow$), have positive flow of those same vehicle types in the candidate equilibrium. By analyzing each case, we show that $\newflow=\prevflow$.

We first present a proposition suggesting that for sufficiently large $P$, no equilibrium flow will exist on roads with toll $P$. We then provide a lemma stating that in any equilibrium in which all vehicle types use the roads they were on in the routing on which the tolls were based, $\prevflow$, then the equilibrium flow will be equal to $\prevflow$. Finally, we examine the case in which $\newflow$ has at least one vehicle type with positive flow on some road or roads which did not have positive flow of the vehicle type in $\prevflow$ (though not on a road with toll $P$). To reiterate, our three cases are as follows:
\begin{enumerate}[nosep]
	\item $\newflow$ has some vehicle type with positive flow on a road which did not have positive vehicle flow of any type in $\prevflow$,
	\item $\newflow$ is such that each vehicle type has positive flow on the same roads as that vehicle type did in $\prevflow$, or \label{cases:case2}
	\item $\newflow$ is such that at least one vehicle type has positive flow on a road which, in $\prevflow$, did not have positive flow of that vehicle type but did have positive flow of another vehicle type. \label{cases:case3}
\end{enumerate} 

The first case is handled in a proposition, the second in a lemma, and the third in the subsequent development.

\begin{proposition}\label{prop:toll_structure}
	For sufficiently large $P$, users in equilibrium will not use a road with toll $P$. Formally, for equilibrium flow $f$ experiencing tolls $\tau(\prevflow)$,
	$$
	f^j_i = 0 \; \forall j \in \vtypeset \land i \in \roadset \setminus \cup_{j' \in {[}m{]} }\sharedroads^\prevflow_{j'} \; .
	$$
\end{proposition}
This means that no equilibrium will have positive flow on roads that did not have positive flow in $z^*$.

\begin{lemma}\label{lma:same_roads_unique}
	Consider the setting of multitype congestion games with affine cost functions on parallel networks. Consider any routing $\prevflow$ that has an associated acyclic graph; the existence of such a routing is provided in Theorem~\ref{thm:routing}. Assume that in routing $\prevflow$, every road has positive flow of some vehicle type. Then levy the following tolls $\tau(\prevflow)$:
	\begin{equation}\label{eq:tolls}
		\tau^j_i(\prevflow) = \mu- \ell_i(\prevflow)
	\end{equation}
	Consider a new equilibrium routing $\newflow$ with flow demand of each type no greater than the demand of that type in $\prevflow$, meaning
	\begin{equation}
		(\forall j \in {[}m{]})[\sum_{i \in {[}n{]}}\prevflow^j_i \le \sum_{i \in {[}n{]}}\newflow^j_i  ] \; .
	\end{equation}
	Moreover, assume that each vehicle type has positive flow in $\newflow$ only on roads that it has positive flow on in $\prevflow$, meaning
	\begin{equation}
		(\forall j \in {[}m{]})[\sharedroads^\newflow_j \subseteq \sharedroads^\prevflow_j] \; .
	\end{equation}
	Then under Assumption~\ref{asmp:increasing_cost}, for every $\mu$,
	\begin{enumerate}[nosep]
		\item Equilibrium routing $\newflow$ is unique, meaning for any other equilibrium $\newflow'$, $$(\forall i\in {[}n{]} \land \forall j \in {[}m{]})[\newflow'^j_i = \newflow^j_i] \; . $$
		\item If $\newflow$ has the same flow demand as $\prevflow$, then $\newflow = \prevflow$. Formally, $$(\forall j \in {[}m{]})[\sum_{i \in {[}n{]}}\newflow^j_i = \sum_{i \in {[}n{]}}\prevflow^j_i  ] \implies \newflow = \prevflow \; .$$
	\end{enumerate}
\end{lemma}

This lemma, which is proved in the appendix, handles Case~\ref{cases:case2} above. We will also use the first statement in proving Case~\ref{cases:case3}, which will be done in the remainder of this proof. 

Let us use $\mathcal{I}$ to denote the set of roads with positive flow in $z^*$:
\begin{equation}\label{eq:used_roads}
	\mathcal{I} = \{ i \in {[}n{]} \; | \; (\exists j \in {[}m{]}) {[} z^{*j}_i>0 {]}  \} \; .
\end{equation}

Next we define the notation $c_i^{\bar{\tau}}(z_i)$, which we call the \emph{standard cost} on a road, to denote the minimum cost on a road to the vehicle types which are on that road -- if it has vehicles which are on it in $z^*$ this road will have their tolled cost, otherwise it will have the large cost of $P$.
\begin{equation}\label{eq:min_cost}
	c_i^{\bar{\tau}}(z_i) :=  \begin{cases} c_i(z_i) + \mu - c_i(z_i^*) & \text{if} \; i \in \mathcal{I} \\
	P & \text{otherwise.} \end{cases}
\end{equation}

\begin{proposition}\label{prop:max_gap}
By the definition of Wardrop Equilibrium (Definition~\ref{def:NE}) and the structure of the tolls, all pairs of used roads will have standard cost within $\epsilon$ of each other. More specifically,
\begin{equation}\label{eq:cost_gap}
	\forall i, i' \in {[}n{]}, j \in {[}m{]} \;  \text{s.t.} \; \hat{z}^j_i > 0, \; c_i^{\bar{\tau}}(\hat{z}_i) \le c_{i'}^{\bar{\tau}}(\hat{z}_{i'}) + \epsilon \; .
\end{equation}
\end{proposition}
\begin{proof}
The reason for this is as follows. For any road $i$ with positive flow of vehicle type $j$, if there is another road with $c_{i'}^{\bar{\tau}}(\hat{z}_{i'}) + \epsilon < c_i^{\bar{\tau}}(\hat{z}_i)$ then $c_{i'}^{\tau_j}(\hat{z}_{i'}) < c_{i}^{\tau_j}(\hat{z}_{i})$, violating Wardrop Equilibrium conditions.
\end{proof}

So far we have established that in a Wardrop Equilibrium, all used roads will have standard cost within $\epsilon$ of each other, and no roads that were not used in $z^*$ will be used in this equilibrium routing. We will next investigate the standard costs on the roads as it relates the vehicle types on each road. 

We now define some additional notation: let $\underline{c}$ denote the minimum standard cost of all roads which were used in $z^*$, and $\overline{c}$ denote the maximum standard of all used roads which were used in $z^*$. Mathematically,
\begin{equation}\label{eq:min_max_cost}
\begin{aligned}
	\underline{c} = \min_{i \in \mathcal{I}} c_i^{\bar{\tau}}(z_i) \\
	\overline{c} = \max_{i \in \mathcal{I}} c_i^{\bar{\tau}}(z_i)
\end{aligned}
\end{equation}

Since by Lemma~\ref{lma:same_roads_unique} if no vehicle types in $\hat{z}$ are on roads they were not on in $z^*$, meaning for all $j \in {[}m{]}$, $\mathcal{N}^{\hat{z}}_j \subseteq \mathcal{N}^{z^*}_j$, then $\hat{z} = z^*$. Accordingly, we restrict our attention to the situation in which some road has a vehicle type in $\hat{z}$ that it did not have in $z^*$, meaning 
\begin{equation}
	\exists j \in {[}m{]}, i \in {[}n{]} \; \text{s.t.} \; \hat{z}^j_i > 0 \land z^{*j}_i = 0
\end{equation}

Next we make a statement about the roads containing vehicle types in $\hat{z}$ that were not there in $z^*$. 

\begin{lemma}\label{lma:relative_costs}
Any road that has a vehicle type on it in $\hat{z}$ but not in $z^*$ will have standard cost $\underline{c}$ under routing $\hat{z}$, and all roads that have this specific vehicle type in $z^*$ have standard cost $\overline{c}$ under routing $\hat{z}$. Mathematically,
\begin{align*}
&(\forall j \in {[}m{]} \land i \in {[}n{]} ) \big[ \hz^j_i > 0 \land z^{*j}_i = 0 \implies c^{\overline{\tau}}_i(\hat{z}_i) = \underline{c}] \; \text{, and} \\
&(\forall j \in {[}m{]} \land i \in {[}n{]} ) \big[ \hz^j_i > 0 \land z^{*j}_i = 0 \implies \\
&\qquad \qquad (\forall i' \in \mathcal{N}^{z^*}_j) [c^{\overline{\tau}}_{i'}(\hat{z}_{i'}) =\underline{c} + \epsilon]  \big] \; .
\end{align*}

Moreover,
\begin{equation}
	\overline{c} = \underline{c} + \epsilon \; .
\end{equation}
\end{lemma}

Since this logic applies to all roads which have positive flow of a vehicle type in $\hat{z}$ that did not have it in $z^*$, we can construct a set of roads which contains all roads with positive flow in $\hat{z}$ that did not have positive flow of those types in $z^*$. This set will contain these roads and possibly other roads as well. We partition $\mathcal{I}$, the roads which can have positive flow in an equilibrium, into three partitions: $\mathcal{I}_1$, $\mathcal{I}_2$, and $\mathcal{I}_3$. We define first two iteratively, and we define the third to be the roads not included in the first two.

\begin{definition}\label{def:partitions}
Let $\mathcal{I}_1$ denote roads which have positive flow of vehicle types in $\hat{z}$ that were not there in $z^*$, and iteratively add in all roads which share common vehicle types (relative to $z^*$) with those roads. We define $\mathcal{I}_2$ iteratively as well, starting with roads which had positive flow in $z^*$ of those vehicle types now in $\mathcal{I}_1$, and adding all roads which share common vehicle types (relative to $z^*$) with those roads. We define $\mathcal{I}_3$ as all remaining roads. To formulate this mathematically, first define a helper function which takes in a set of roads and outputs the union of that set and all roads that share vehicle types with it in $z^*$:
\begin{equation}
	\Phi(A) = A \cup \{ i \in {[}n{]} \; | \; (\exists j \in {[}m{]}, i' \in A){[}z^{*j}_{i'}>0 \land z^{*j}_{i}>0  {]}  \} \; .
\end{equation}

Then,
\begin{equation}\label{eq:partitions_definition}
\begin{aligned}
	&\mathcal{I}_1 = \Phi^n( \{ i \in \mathcal{I} \; | \; (\exists j \in {[}m{]}) {[} \hat{z}^j_i > 0 \land z^{*j}_i = 0  {]}  \} ) \\
	&\mathcal{I}_2 =\Phi^n( \{ i' \in \mathcal{I} \; | \; (\exists j \in {[}m{]}, i \in {[}n{]}) \\
	&\qquad \qquad {[} \hat{z}^j_{i} > 0 \land z^{*j}_{i} = 0 \land i' \in \mathcal{N}^{z^*}_{j} {]}  \} ) \\
	&\mathcal{I}_3 = \mathcal{I} \setminus (\mathcal{I}_1 \cup \mathcal{I}_2) \; .
\end{aligned}
\end{equation}
\end{definition}

The construction of $\mathcal{I}_1$ and $\mathcal{I}_2$ to have roads with continuously overlapping vehicles types with respect to $z^*$, implying that all its roads have the same standard cost. Using this and Lemma~\ref{lma:relative_costs} we can state the following:
\begin{equation}\label{eq:standard_cost_partitions}
\begin{aligned}
	&\forall i \in \mathcal{I}_1  : \; c^{\overline{\tau}}_i(\hat{z}_i) = \underline{c} \\
	&\forall i \in \mathcal{I}_2  : \; c^{\overline{\tau}}_i(\hat{z}_i) = \overline{c} = \underline{c} + \epsilon  \\
	& \forall i \in \mathcal{I}_3  : \; \underline{c} \le c^{\overline{\tau}}_i(\hat{z}_i) \le \overline{c}  \; .
\end{aligned}
\end{equation}

If \eqref{eq:standard_cost_partitions} and \eqref{eq:partitions_definition} share any roads, then $\underline{c} = \underline{c} + \epsilon$, yielding a contradiction and proving the theorem. Otherwise, these three sets are partitions of $\mathcal{I}$ and do not share any roads in common with each other. Moreover, these roads partition the vehicle types in $z^*$, meaning that in $z^*$, there are no vehicle types which have positive flow on more than one of the sets $\mathcal{I}_1$, $\mathcal{I}_2$, or $\mathcal{I}_3$.

Moreover, as an implication of the definition, all roads in $\mathcal{I}_3$ have $\hat{z}^j_i = z^{*j}_i$. Accordingly, by the definition of the tolls,
\begin{equation}
	\forall i \in \mathcal{I}_3 : \; c^{\overline{\tau}}_i(\hat{z}_i) = \mu  \; .
\end{equation}

We next relate $\underline{c}$ and $\overline{c}$ to $\mu$. First we deal with $\mathcal{I}_1$, which has $c^{\overline{\tau}}_i(\hat{z}_i) = \underline{c}$. 

\begin{lemma}\label{lma:std_cost_greater_flow}
	For any road $i$ in partition $\mathcal{I}_1$, constructed as in Definition~\ref{def:partitions}, the standard cost on $i$ will be at least as great in $\hat{z}$ as it is in $z^*$. Mathematically,
	\begin{equation}
		(\forall i \in \mathcal{I}_1)[c^{\overline{\tau}}_i(\hat{z}_i) \ge c^{\overline{\tau}}_i(z^*_i)]
	\end{equation}
\end{lemma}

We now deal with the roads in $\mathcal{I}_2$. 

\begin{lemma}\label{lma:std_cost_lesser_flow}
	For any road $i$ in partition $\mathcal{I}_2$, constructed as in Definition~\ref{def:partitions}, the standard cost on $i$ will not be greater in $\hat{z}$ as it is in $z^*$. Mathematically,
	\begin{equation}
		(\forall i \in \mathcal{I}_2)[c^{\overline{\tau}}_i(\hat{z}_i) \le c^{\overline{\tau}}_i(z^*_i)]
	\end{equation}
\end{lemma}

However, there is a contradiction between Lemmas~\ref{lma:std_cost_greater_flow} and \ref{lma:std_cost_lesser_flow} (combined with the fact that in $c^{\overline{\tau}}_i(z^*_i)=\mu$ for all roads $i$ in $\mathcal{I}$) with \eqref{eq:standard_cost_partitions}, proving Case~\ref{cases:case3}. With all three cases handled, the theorem is proved.

\end{proof}

The theorem above provides tolls which ensure that any equilibrium will minimize the social cost of routing. Importantly, the tolls for each vehicle type differ by at most $\epsilon$, which must be greater than zero, but can be arbitrarily small. A system designer can consider different incentive methods to create this small differentiation. Some possible methods include actually differentiated tolls, randomly providing gifts to vehicles of some types to differentiate the tolls on expectation, or even pro-social public messaging. This small differentiation addresses many of the problems with non-anonymous tolls, including fairness and privacy.

\section{VARIABLE MARGINAL COST TOLLS}
\label{sct:marginal_cost}
The tolls provided in this paper have been \emph{fixed cost} tolls, meaning that the tolls are calculated based on the network structure and overall flow demand, but the toll on a link does not vary with the flow on that link. However, previous works have shown that fixed tolls are not \emph{strongly robust} to mischaracterizations of latency functions in a network, meaning that if the information used to calculate the tolls is not accurate, the tolls do not necessarily incentivize optimal behavior \cite{brown2017studies}. Moreover, the tolls provided depend on knowledge of the entire network and the flow demand of all vehicle types. A class of tolls which alleviates this latter requirement is \emph{variable marginal cost tolls}, in which the toll provided to vehicles of type $j$ on road $i$ is $$\tau_i^j(z_i) = (\sum_{j \in {[}m{]}}z^j_i)\frac{\partial}{\partial z^j_i}c_i(z_i) \; .$$

Assuming, as we do, that all vehicle types experience latency identically, if the social cost function is convex, then variable marginal cost tolling ensure that any resulting equilibrium has the social cost of the socially optimal routing\footnote{This is because the tolls align the Wardrop Equilibrium conditions and the conditions for first-order optimality of the social cost.} \cite{sandholm2005negative}. However, in our setting the assumption of convex social cost function does not hold. It is therefore worth investigating whether marginal cost tolls are optimal in the setting of mixed autonomy; the following proposition shows that they are not.

\begin{proposition}\label{prop:marginal_cost}
Consider the setting of multitype congestion games with affine cost functions. The Price of Anarchy, with variable marginal cost tolls applied, is lower bounded by $k/(2 \sqrt{k} -1)$, where $k$ is the maximum degree of asymmetry.
\end{proposition}
\begin{proof}

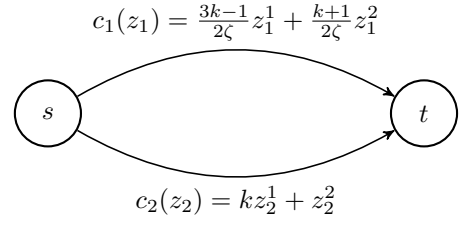
\begin{figure}
	\centering
	\begin{tikzpicture}[baseline={([yshift={-\ht\strutbox}]current bounding box.north)}, ->, >=stealth', auto, semithick, node distance=5cm]
			\tikzstyle{every state}=[fill=white,draw=black,thick,text=black,scale=1]
			\node[state]    (0)               {$s$};
			\node[state]    (1)[right of=0]   {$t$};
			\path
			(0) edge[bend left]		node[above]{$c_1(z_1) =  \frac{3k-1}{2\zeta} z^1_1 + \frac{k+1}{2\zeta} z^2_1$}     (1)
			(0) edge[bend right]	node[below]{$c_2(z_2) = k z^1_2 + z^2_2$}     (1);
	\end{tikzpicture}
	\caption{Examples proving Proposition~\ref{prop:marginal_cost}. Consider $1$ unit flow demand of vehicle type $1$ and $\zeta$ units flow demand of vehicles type $2$.}
	\label{fig:marginal_cost_proof}
\end{figure}

We prove this proposition by example. Consider the network shown in Figure~\ref{fig:marginal_cost_proof}. Consider $1$ unit flow demand of vehicle type $1$ and $\zeta$ units flow demand of vehicles type $2$.

The form of the marginal cost tolls make it such that the Wardrop Equilibrium conditions are satisfied when first-order optimality conditions are satisfied. We will therefore lower bound the Price of Anarchy by taking the ratio of the social cost at a saddle point of the social cost function to the social cost at another feasible routing. Since the socially optimal routing will not have greater social cost than this latter term, this will be a lower bound on the PoA.

With the parameters chosen above, a saddle point of the social cost function (and therefore a Wardrop Equilibrium with tolls applied) exists at $z^1_1=0$, $z^2_1=\zeta$, $z^1_2=1$, $z^2_2=0$. Let us use $\hat{z}$ to denote this routing. Let us use $\tilde{z}$ to denote another routing, $z^1_1=1$, $z^2_1=0$, $z^1_2=0$, $z^2_2=\zeta$. Then, for any socially optimal routing $z^*$,
\begin{equation*}
\frac{C(\hat{z})}{C(z^*)} \ge \frac{C(\hat{z})}{C(\tilde{z})} = \frac{(k+1)\zeta^2 +2k \zeta  }{2\zeta^3 + 3 k -1 } \; .
\end{equation*}
If we choose $\zeta = \sqrt{k}$, we find
\begin{equation*}
\frac{C(\hat{z})}{C(z^*)} \ge \frac{k }{2 \sqrt{k}-1} \; ,
\end{equation*}
meaning that we have a lower bound on the Price of Anarchy which is unbounded with increasing $k$.
\end{proof}

This proposition shows that marginal cost tolls do not necessarily induce optimal equilibria in the setting of mixed autonomy. Moreover, we showed that a lower bound on the Price of Anarchy with variable marginal cost tolls applied scales with $\sqrt{k}$, where $k$ is the maximum degree of asymmetry.

\section{CONCLUSION AND FUTURE WORK}
\label{sct:conclusion}
In this work we explored the role of differentiation in tolling for traffic networks with mixed autonomy. We show via a clear proof mechanism that our provided differentiated tolls yield socially optimal routing. We extend previous Price of Anarchy results to the setting of more than two vehicle types, provide anonymous tolls with performance bounds, and compare these bounds to the bounds on untolled networks and networks tolled with anonymous tolls. We then provided tolls which are only infinitesimally differentiated, but in parallel networks can induce an optimal routing. Finally, we quantified the limitations of variable marginal cost tolling in our setting.

There are a number of important directions for continuing this line of work. There is a gap between the upper bound performance guarantees of our anonymous tolls and the lower bounds that are established via example. Aside from infinitesimally differentiated tolling, there are other ways to achieve some of the goals of anonymous tolling without sacrificing network performance. Specifically, one can consider a scheme in which different vehicle classes are clustered together for the purposes of tolling, and a system designer wishes to determine the optimal clustering and tolling. Continuing this direction holds promise for understanding the tradeoffs in choosing tolling granularity in the presence of many different vehicle types, and choosing a toll method that can be efficiently and fairly implemented, while preserving traffic network performance efficiency.

%%%%%%%%%%%%%%%%%%%%%%%%%%%%%%%%%%%%%%%%%%%%%%%%%%%%%%%%%%%%%%%%%%%%%%%%%%%%%%%%

\bibliographystyle{IEEEtran}
\bibliography{lazar.bib}

% Generated by IEEEtran.bst, version: 1.14 (2015/08/26)
\begin{thebibliography}{10}
\providecommand{\url}[1]{#1}
\csname url@samestyle\endcsname
\providecommand{\newblock}{\relax}
\providecommand{\bibinfo}[2]{#2}
\providecommand{\BIBentrySTDinterwordspacing}{\spaceskip=0pt\relax}
\providecommand{\BIBentryALTinterwordstretchfactor}{4}
\providecommand{\BIBentryALTinterwordspacing}{\spaceskip=\fontdimen2\font plus
\BIBentryALTinterwordstretchfactor\fontdimen3\font minus
  \fontdimen4\font\relax}
\providecommand{\BIBforeignlanguage}[2]{{%
\expandafter\ifx\csname l@#1\endcsname\relax
\typeout{** WARNING: IEEEtran.bst: No hyphenation pattern has been}%
\typeout{** loaded for the language `#1'. Using the pattern for}%
\typeout{** the default language instead.}%
\else
\language=\csname l@#1\endcsname
\fi
#2}}
\providecommand{\BIBdecl}{\relax}
\BIBdecl

\bibitem{noguchi2017npr}
\BIBentryALTinterwordspacing
Y.~Noguchi, \emph{Self-Driving Cars Could Ease Our Commutes, But That'll Take A
  While}.\hskip 1em plus 0.5em minus 0.4em\relax National Public Radio, 2017.
  [Online]. Available:
  \url{https://www.npr.org/sections/alltechconsidered/2017/02/10/514091049/self-driving-cars-could-ease-our-commutes-but-thatll-take-a-while}
\BIBentrySTDinterwordspacing

\bibitem{lioris2017platoons}
J.~Lioris, R.~Pedarsani, F.~Y. Tascikaraoglu, and P.~Varaiya, ``Platoons of
  connected vehicles can double throughput in urban roads,''
  \emph{Transportation Research Part C: Emerging Technologies}, vol.~77, pp.
  292--305, 2017.

\bibitem{mehr2019will}
N.~Mehr and R.~Horowitz, ``How will the presence of autonomous vehicles affect
  the equilibrium state of traffic networks?'' \emph{IEEE Transactions on
  Control of Network Systems}, vol.~7, no.~1, pp. 96--105, 2019.

\bibitem{lazar2020routing}
D.~A. Lazar, S.~Coogan, and R.~Pedarsani, ``Routing for traffic networks with
  mixed autonomy,'' \emph{IEEE Transactions on Automatic Control}, 2020.

\bibitem{pigou1932economics}
A.~C. Pigou, ``The economics of welfare,'' \emph{McMillan\&Co., London}, 1920.

\bibitem{beckmann1956studies}
M.~Beckmann, C.~B. McGuire, and C.~B. Winsten, ``Studies in the economics of
  transportation,'' Tech. Rep., 1956.

\bibitem{dafermos1973toll}
S.~C. Dafermos, ``Toll patterns for multiclass-user transportation networks,''
  \emph{Transportation science}, vol.~7, no.~3, pp. 211--223, 1973.

\bibitem{mehr2019pricing}
N.~Mehr and R.~Horowitz, ``Pricing traffic networks with mixed vehicle
  autonomy,'' in \emph{2019 American Control Conference (ACC)}.\hskip 1em plus
  0.5em minus 0.4em\relax IEEE, 2019, pp. 2676--2682.

\bibitem{lazar2020optimal}
D.~A. Lazar and R.~Pedarsani, ``Optimal tolling for multitype mixed autonomous
  traffic networks,'' \emph{IEEE Control Systems Letters}, 2020.

\bibitem{iqbal2008designing}
M.~U. Iqbal and S.~Lim, ``Designing tolling technologies with privacy in mind:
  A user perspective,'' \emph{Transportation Research Part C: Emerging
  Technologies}, vol.~61, no.~2, pp. 1--25, 2008.

\bibitem{rossger2009motivational}
L.~R{\"o}{\ss}ger, J.~Schade, and T.~Tretvik, ``Motivational factors
  influencing behavioural responses to charging measures in freight operator
  sector,'' 2009.

\bibitem{brown2017studies}
P.~N. Brown and J.~R. Marden, ``Studies on robust social influence mechanisms:
  Incentives for efficient network routing in uncertain settings,'' \emph{IEEE
  Control Systems Magazine}, vol.~37, no.~1, pp. 98--115, 2017.

\bibitem{sandholm2002evolutionary}
W.~H. Sandholm, ``Evolutionary implementation and congestion pricing,''
  \emph{The Review of Economic Studies}, vol.~69, no.~3, pp. 667--689, 2002.

\bibitem{dafermos1969traffic_general}
S.~C. Dafermos and F.~T. Sparrow, ``The traffic assignment problem for a
  general network,'' \emph{Journal of Research of the National Bureau of
  Standards B}, vol.~73, no.~2, pp. 91--118, 1969.

\bibitem{wardrop1900some}
J.~Wardrop, ``Some theoretical aspects of road traffic research,'' in
  \emph{Inst Civil Engineers Proc London/UK/}, 1900.

\bibitem{depalma1998optimization}
A.~De~Palma and Y.~Nesterov, ``Optimization formulations and static equilibrium
  in congested transportation networks,'' Tech. Rep., 1998.

\bibitem{roughgarden2002bad}
T.~Roughgarden and {\'E}.~Tardos, ``How bad is selfish routing?'' \emph{Journal
  of the ACM (JACM)}, vol.~49, no.~2, pp. 236--259, 2002.

\bibitem{correa2008geometric}
J.~R. Correa, A.~S. Schulz, and N.~E. Stier-Moses, ``A geometric approach to
  the price of anarchy in nonatomic congestion games,'' \emph{Games Econ.
  Behavior}, vol.~64, no.~2, pp. 457--469, 2008.

\bibitem{brown2019tragedy}
P.~N. Brown, ``A tragedy of autonomy. self-driving cars and urban congestion
  externalities,'' in \emph{2019 57th Annual Allerton Conference on
  Communication, Control, and Computing (Allerton)}.\hskip 1em plus 0.5em minus
  0.4em\relax IEEE, 2019, pp. 981--986.

\bibitem{sandholm2005negative}
W.~H. Sandholm, ``Negative externalities and evolutionary implementation,''
  \emph{The Review of Economic Studies}, vol.~72, no.~3, pp. 885--915, 2005.

\bibitem{brown2016study}
P.~N. Brown and J.~R. Marden, ``A study on price-discrimination for robust
  social coordination,'' in \emph{2016 American Control Conference
  (ACC)}.\hskip 1em plus 0.5em minus 0.4em\relax IEEE, 2016, pp. 1699--1704.

\bibitem{wollenstein2021routing}
S.~Wollenstein-Betech, M.~Salazar, A.~Houshmand, M.~Pavone, I.~C. Paschalidis,
  and C.~G. Cassandras, ``Routing and rebalancing intermodal autonomous
  mobility-on-demand systems in mixed traffic,'' \emph{IEEE Transactions on
  Intelligent Transportation Systems}, 2021.

\bibitem{amin2021efficient}
S.~Amin, P.~Jaillet, and M.~Wu, ``Efficient carpooling and toll pricing for
  autonomous transportation,'' \emph{arXiv preprint arXiv:2102.09132}, 2021.

\bibitem{zhu2019routing}
Y.~Zhu and K.~Savla, ``On routing drivers through persuasion in the long run,''
  in \emph{2019 IEEE 58th Conference on Decision and Control (CDC)}.\hskip 1em
  plus 0.5em minus 0.4em\relax IEEE, 2019, pp. 4091--4096.

\bibitem{wu2021value}
M.~Wu, S.~Amin, and A.~E. Ozdaglar, ``Value of information in bayesian routing
  games,'' \emph{Operations Research}, vol.~69, no.~1, pp. 148--163, 2021.

\bibitem{kreidieh2018dissipating}
A.~R. Kreidieh, C.~Wu, and A.~M. Bayen, ``Dissipating stop-and-go waves in
  closed and open networks via deep reinforcement learning,'' in \emph{2018
  21st International Conference on Intelligent Transportation Systems
  (ITSC)}.\hskip 1em plus 0.5em minus 0.4em\relax IEEE, 2018, pp. 1475--1480.

\bibitem{gunter2020commercially}
G.~Gunter, D.~Gloudemans, R.~E. Stern, S.~McQuade, R.~Bhadani, M.~Bunting,
  M.~L. Delle~Monache, R.~Lysecky, B.~Seibold, J.~Sprinkle \emph{et~al.}, ``Are
  commercially implemented adaptive cruise control systems string stable?''
  \emph{IEEE Transactions on Intelligent Transportation Systems}, 2020.

\bibitem{lazar2019optimal}
D.~A. Lazar, S.~Coogan, and R.~Pedarsani, ``Optimal tolling for heterogeneous
  traffic networks with mixed autonomy,'' in \emph{58th Conference on Decision
  and Control (CDC)}.\hskip 1em plus 0.5em minus 0.4em\relax IEEE, 2019, pp.
  4103--4108.

\bibitem{bureau1964manual}
``Traffic assignment manual,'' \emph{Bureau of Public Roads}, 1964.

\bibitem{lazar2017routing}
D.~A. Lazar, S.~Coogan, and R.~Pedarsani, ``Capacity modeling and routing for
  traffic networks with mixed autonomy,'' in \emph{56th Annual Conference on
  Decision and Control (CDC)}.\hskip 1em plus 0.5em minus 0.4em\relax IEEE,
  2017.

\bibitem{smith1979existence}
M.~J. Smith, ``The existence, uniqueness and stability of traffic equilibria,''
  \emph{Transportation Research Part B: Methodological}, vol.~13, no.~4, pp.
  295--304, 1979.

\end{thebibliography}

\section{APPENDIX}
\label{sct:appendix}
We will first present additional lemmas and a proposition to aid the proof of the theorems. We will then prove the theorems and lemmas.

\subsection{Lemmas and Proposition}
\begin{lemma}\label{lma:poa1}
	For any feasible equilibrium routing $\hz$ and optimal routing $\sz$ in an untolled network with maximum degree of asymmetry $k<4$,
	$$C(\hz) \le 4/(4-k) C(\sz) \; .$$
\end{lemma}

For the next proposition and lemmas, we define the following functions. Consider an original network with latency function $c$, and some specific equilibrium flow for this network, $\hz$. We first define a flow aggregator function which combines the flow in $\hz$ of each type on a road. We define $\hf : \mathbb{R}^m_{\ge 0} \rightarrow \mathbb{R}_{\ge 0}$, where
\begin{align}\label{eq:hf}
	\hf_i(\hz_i) = \sum_{j \in {[}m{]} } \hz^j_i \; .
\end{align}

We next define a latency function parameterized by another latency function and a flow. First define an intermediary function $v_i^{j'} : \mathbb{R}^m_{\ge 0} \rightarrow \mathbb{R}_{\ge 0}$ such that
\begin{align}\label{eq:vj}
	v^{j'}_i(\hz_i) = \begin{cases}
	0 & j = 0 \\
	\sum_{j \in {[}j'{]}} \hz^j_i & j' \in {[}m-1{]} \\
	\infty & j = m
	\end{cases}
\end{align}

We then define a new latency function, parameterized by a flow and another cost function, $\ell_i : \mathbb{R}_{\ge 0} \times \mathbb{R}^m_{\ge 0} \times \mathcal{C}  \rightarrow \mathbb{R}_{\ge 0}$ as
\begin{align}
	&\ell_i(f_i;\hz_i, c_i) \nonumber \\
	& \hspace{1cm} := \begin{cases} b_i + \sum_{j \in {[}j'-1{]}} a_i^j \hz^i_j+ a^{j'}_i(f_i - \sum_{j \in {[}j'-1{]}} \hz^i_j) \end{cases} \nonumber \\
	& \hspace{4cm} \text{for } v^{j'-1}_i \le f_i \le v^{j'}_i \label{eq:ell} \; ,
\end{align}
And similarly define $\ell$ as a column vector where its $i$th entry is $\ell_i$. In essence, we are defining a new function which treats flow as flow of type $1$ of $c_i$ until reaching $\hz^1_i$, then treats it as type two until reaching $\hz^2_i + \hz^1_i$, and so on. All remaining flow after exceeding the sum of flows in $\hz_i$ are treated as type $m$. We similarly define a social cost function $L$ such that
\begin{align}\label{eq:L}
	L(f;\hz,c) := \sum_{i \in {[}n{]} } f_i \ell_i(f_i;\hz_i, c_i) \; .
\end{align}

\begin{proposition}\label{prop:agg}
	For a network with latency function $c$ and Wardrop Equilibrium $\hz$, and $\hf$, $v^{j'}_i$, $\ell$, and $L$ defined as in \cref{eq:hf,eq:vj,eq:ell,eq:L},
\begin{enumerate}
	\item $L(\hf(\hz);\hz,c) = C(\hz)$, and
	\item if $\hz$ is a Wardrop Equilibrium for $c$, then $\hf$ is a Wardrop Equilibrium for $\ell$.
\end{enumerate}
\end{proposition}

\begin{lemma}\label{lma:agg_poa}
	For any Wardrop Equilibrium flow $\hf$ and socially optimal flow $\fs$ of a network with cost function $L(\cdot;\hz,c)$, $$L(\hf;\hz,c) \le 4/3 L(\fs;\hz,c) \; .$$
\end{lemma}

\begin{lemma}\label{lma:agg_opt_cost}
	For socially optimal flow $\sz$ of a network with cost function $C$ and socially optimal flow $\fs$ relative to a network with cost function $L(\fs;\hz,c)$ (where $C$ has the same latency function as the argument to $L$), if $c$ has maximum degree of asymmetry $k$, then $$L(\fs;\hz,c) \le k C(\sz) \; .$$
\end{lemma}

\subsection{Proofs of Theorems and Lemmas}

\emph{Proof of Theorem~\ref{thm:poa}.}

We prove this theorem by putting together Proposition~\ref{prop:agg} with Lemmas~\ref{lma:agg_poa} and \ref{lma:agg_opt_cost}, and combining it with Lemma~\ref{lma:poa1} as follows. We begin with a network with cost function $c$ with maximum degree of asymmetry $k$. By Proposition~\ref{prop:agg}, for any Wardrop Equilibrium $\hz$ we can construct cost function $L(f;\hz,c_i)$ with equilibrium flow $\hf(\hz)$ such that $$C(\hz) = L(\hf(\hz);\hz,c) \; .$$ By Lemma~\ref{lma:agg_poa}, for optimal routing $\fs$ for $L$, $$L(\hf(\hz);\hz,c) \le 4/3 L(\fs;\hz,c) \; , $$ and by Lemma~\ref{lma:agg_opt_cost}, for optimal routing $\sz$ relative to the original cost function $c$, $$L(\fs;\hz,c) \le k C(\sz) \; .$$ Thus, $$C(\hz) \le 4k/3 C(\sz) \; .$$ Combined with Lemma~\ref{lma:poa1}, this proves the theorem. \hfill \qed

\emph{Proof of Lemma~\ref{lma:poa1}.}
	We prove the Lemma using a tool from \cite{lazar2020routing} which extends the technique in \cite{correa2008geometric} to non-monotone cost functions. From Lemma~1 of \cite{lazar2020routing}, we have, for Wardrop Equilibrium $\hz$ and optimal routing $\sz$,
\begin{align}
	&C(\hz) \le \frac{1}{1-\beta(\mathcal{C})}C(\sz) \; , \text{where} \nonumber \\
	&\beta(\mathcal{C}) := \sup_{c \in \mathcal{C}, \hz \in \mathbb{R}^{nm}_{\ge 0}, z \in \mathcal{Z}} \frac{\langle c(\hz) - c(z), z \rangle}{\langle c(\hz), \hz \rangle} \; , \label{eq:poa_mechanics}
\end{align}
where $\mathcal{Z}$ denotes the set of feasible routings. We use $\mathcal{C}_k$ to denote the set of latency functions with maximum degree of asymmetry bounded by $k$. Similarly, let us use $\mathcal{A}_k$ to denote the set of vectors in $\mathbb{R}^m_{\ge 0}$ such that the ratio of any two elements is less than $k$. Then,
\begin{align*}
	&\beta(\mathcal{C}_k) \le \sup_{c_i \in \mathcal{C}_k, \hz_i, z_i  \in \mathbb{R}^{m}_{\ge 0}} \frac{\langle c_i(\hz_i) - c_i(z_i), z_i \rangle}{\langle c_i(\hz_i),\hz_i \rangle} \\
	&=\sup_{a_i \in \mathcal{A}_k, b_i \in \mathbb{R}_{\ge 0},\hz_i, z_i  \in \mathbb{R}^{m}_{\ge 0}}  \frac{(\sum_{j \in {[}m{]}}z^j_i)(\sum_{j \in {[}m{]}}a^j_i(\hz^j_i -z^j_i)) }{(\sum_{j \in {[}m{]}}\hz^j_i)(b_i + \sum_{j \in {[}m{]}}a^j_i \hz^j_i)} \\
	& \le \sup_{a_i \in \mathcal{A}_k,\hz_i, z_i  \in \mathbb{R}^{m}_{\ge 0}}  \frac{(\sum_{j \in {[}m{]}}z^j_i)(\sum_{j \in {[}m{]}}a^j_i(\hz^j_i -z^j_i)) }{(\sum_{j \in {[}m{]}}\hz^j_i)(\sum_{j \in {[}m{]}}a^j_i \hz^j_i)} \; .
\end{align*}
where the first inequality results from all terms in the denominator being nonnegative. Without loss of generality, we assign vehicle type labels such that on the road $i$ being considered, $a^1_i \ge a^2_i \ge \hdots \ge a^m_i$. Then, 
\begin{align*}
	\beta(\mathcal{C}_k) & \le \sup_{a_i \in \mathcal{A}_k,\hz_i, z_i  \in \mathbb{R}^{m}_{\ge 0}}  \frac{a^1_i(\sum_{j \in {[}m{]}}z^j_i)(\sum_{j \in {[}m{]}}\hz^j_i -z^j_i) }{a^m_i(\sum_{j \in {[}m{]}}\hz^j_i)(\sum_{j \in {[}m{]}}\hz^j_i)} \\
	& \le \sup_{\hz_i, z_i  \in \mathbb{R}^{m}_{\ge 0}}  k \frac{(\sum_{j \in {[}m{]}}z^j_i)(\sum_{j \in {[}m{]}}\hz^j_i)-(\sum_{j \in {[}m{]}}z^j_i)^2 }{(\sum_{j \in {[}m{]}}\hz^j_i)^2} \\
	& = \sup_{\hz^j_i, z^j_i  \in \mathbb{R}_{\ge 0}}  k \frac{z^j_i \hz^j_i -(z^j_i)^2 }{(\hz^j_i)^2} \\
	& = k/4 \; ,
\end{align*}
where the first inequality results from the expression being nonnegative, and the last expression is found by maximizing with respect to $z^j_i$, with $z^j_i$ and $\hz^j_i$ both greater than zero. With this, we find that for $k<4$, $\Lambda(k) = 4/(4-k)$. \hfill \qed

\emph{Proof of Lemma~\ref{lma:agg_poa}.}
	In this lemma we bound the Price of Anarchy of the latency function defined in  \cref{eq:hf,eq:vj,eq:ell,eq:L}. To do this, we apply \eqref{eq:poa_mechanics} to $L(\hf;\hz,c)$. Let us define $\mathcal{L}$ as the set of functions $L$, and let us define $\mathcal{F}$ as the set of feasible routings for $L$. Then,
\begin{align*}
	&\beta(\mathcal{L}) = \sup_{c \in \mathcal{C}, \hz \in \mathbb{R}^{nm}_{\ge 0}, f \in \mathcal{F}} \frac{\langle \ell(\hf(\hz);\hz,c) - \ell(f;\hz,c), f \rangle}{\langle \ell(\hf(\hz);\hz,c), \hf(\hz) \rangle} \\
	&\le \sup_{c_i \in \mathcal{C}, \hz_i \in \mathbb{R}^{m}_{\ge 0}, f_i \ge 0} \frac{(\ell_i(\hf_i(\hz_i);\hz_i,c_i) - \ell_i(f_i;\hz_i,c_i)) f_i }{(\ell_i(\hf_i(\hz_i);\hz_i,c_i)) \hf_i(\hz_i)} \\
	& \le \sup_{\substack{a_i \in \mathbb{R}^m_{\ge 0} \\ \hz_i \in \mathbb{R}^{m}_{\ge 0} \\ f_i \ge 0} } \frac{f_i\big(\sum\limits_{ j \in {[}m{]} }a^j_i \hz^j_i - \hspace{-6pt} \sum\limits_{j \in {[}j'-1{]} } \hspace{-4pt} a^j_i \hz^j_i - a^{j'}_i(f- \hspace{-8pt} \sum\limits_{j \in {[}j'-1{]}} \hspace{-4pt} \hz^j_i)   \big) }{( \sum\limits_{ j \in {[}m{]} } \hz^j_i )\sum\limits_{ j \in {[}m{]} }a^j_i \hz^j_i} \\
	& = \sup_{a_i \in \mathbb{R}^m_{\ge 0}, \hz_i \in \mathbb{R}^{m}_{\ge 0}, f_i \ge 0 } \frac{f_i \big( \sum\limits_{ j =j' }^m a^j_i \hz^j_i  - a^{j'}_i f +a^{j'}_i\sum\limits_{j \in {[}j'-1{]}} \hz^j_i \big) }{( \sum_{ j \in {[}m{]} } \hz^j_i )\sum_{ j \in {[}m{]} }a^j_i \hz^j_i} \\
	&\le  \sup_{a_i \in \mathbb{R}^m_{\ge 0}, \hz_i \in \mathbb{R}^{m}_{\ge 0}, f_i \ge 0 } \frac{f_i \big( a^{j'}_i\sum\limits_{ j =j' }^m  \hz^j_i  - a^{j'}_i f +a^{j'}_i\sum\limits_{j \in {[}j'-1{]}} \hz^j_i \big) }{( \sum\limits_{ j \in {[}m{]} } \hz^j_i )(a^{j'}_i \sum\limits_{ j \in {[}j'-1{]} } \hz^j_i + a^{j'}_i\sum\limits_{ j =j' }^m \hz^j_i ) } \\
	&=  \sup_{ \hz_i \in \mathbb{R}^{m}_{\ge 0}, f_i \ge 0 } \frac{f_i(\sum_{ j \in {[}m{]} }  \hz^j_i  - f ) }{\big( \sum_{ j \in {[}m{]} } \hz^j_i\big)^2 } \; ,
\end{align*}
where the nonnegative terms in the summation in the inner product in the denominator implies the first inequality, the second inequality stems from choosing $b_i=0$. To achieve the final inequality, we, without loss of generality, choose the vehicle indices order for road $i$ such that $a^1_i \ge a^2_i \ge \hdots \ge a^m_i$. This way, setting $a^j_i = a^{j'}_i$ for $j < j'$ minimizes the denominator, and setting $a^j_i = a^{j'}_i$ for $j \ge j'$ maximizes a term common to the numerator and denominator. We then maximize the expression with respect to $f$, choosing it to be equal to $1/2 \sum_{ j \in {[}m{]} } \hz^j_i$, so 
\begin{equation*}
	\beta(\mathcal{L}) \le 1/4 \; .
\end{equation*}

This proves the lemma. \hfill \qed

\emph{Proof of Lemma~\ref{lma:agg_opt_cost}.}

First define $w_i : \mathbb{R}^m_{\ge 0} \times \mathcal{C}$ such that
\begin{equation*}
	[w_i(z_i;c_i)]_{j'} = \begin{cases}
	\sum_{j \in {[}m{]}}z^j_i & \text{if } j'=\argmin_{j \in {[}m{]}}a^j_i \\
	0 & \text{otherwise}
	\end{cases} \; .
\end{equation*}
In words, $w_i$ puts all the flow in $z_i$ as the vehicle type that congests the least according to cost function $c_i$. Then, for $\fs$ as an optimal routing with respect to $L(f;\hz,c)$ and $\sz$ as an optimal routing with respect to $C$, 
\begin{align*}
	L(\fs; \hz,c) &= \sum_{i \in {[}n{]}}f^*_i \ell_i(f^*_i;\hz_i,c_i) \\
	&\le \sum_{i \in {[}n{]}} \sum_{j \in {[}m{]}} \sz^j_i \ell_i(\sum_{j \in {[}m{]}} \sz^j_i;\hz_i, c_i) \\
	&\le  \sum_{i \in {[}n{]}} \sum_{j \in {[}m{]}} \sz^j_i k c_i(w_i(z^*_i;c_i)) \\
	&\le  \sum_{i \in {[}n{]}} \sum_{j \in {[}m{]}} \sz^j_i k c_i(z^*_i) = k C(\sz) \; ,
\end{align*}
where the first inequality stems from $f^*$ as a minimizer of the cost with respect to $L$, the second inequality is because $k c_i(w_i(\sz_i;c_i)) \ge b_i + \sum_{j \in {[}m{]}}z^{*j}_i \max_{j\in{[}m{]}} a^j_i \ge \ell_i(\sum_{j \in {[}m{]}} \sz^j_i;\hz_i, c_i)$, and the third inequality is from the construction of $w_i$.

Thus, $L(\fs; \hz,c) \le k C(\sz)$, proving the lemma. \hfill \qed

\emph{Proof of Theorem~\ref{thm:anonymous}.}
To prove this theorem, we use the aggregate cost function provided in Proposition~\ref{prop:agg}, then prove a Price of Anarchy bound on this aggregate cost function with tolls applied, and then bound the cost of the optimal routing with respect to the aggregate cost function as compared to the original cost function using Lemma~\ref{lma:agg_opt_cost}.

Accordingly, we first define $\hf$, $v^{j'}_i$, $\ell$, and $L$ as in \cref{eq:hf,eq:vj,eq:ell,eq:L}, and by Prop.~\ref{prop:agg}, $C(\hz) = L(\hf(\hz);\hz,c)$, and $\hf(\hz)$ is a Wardrop Equilibrium for $L(\cdot;\hz,c)$. In an abuse of notation, we redefine $\tau \in \mathbb{R}^n_{\ge 0}$ as a column vector with elements $\tau_i \in \mathbb{R}_{\ge 0}$. We then define
\begin{align*}
	&\ell^{\tau_i}_i(f_i;\hz_i,c_i) := \ell_i(f_i;\hz_i,c_i) + \tau_i \; \text{and}  \\
	&[\ell^{\tau}(f;\hz,c)]_i := \ell^{\tau_i}_i(f_i;\hz_i,c_i) \; .
\end{align*}

Then,
\begin{align*}
	L(\hf(\hz);\hz,c) &= \langle \ell(\hf(\hz);\hz,c),\hf(\hz) \rangle \\
	&= \langle \ell^{\tau}(\hf(\hz);\hz,c),\hf(\hz) \rangle - \tau^T \hf(\hz) \\
	&\le \langle \ell^{\tau}(\hf(\hz);\hz,c), \fs \rangle - \tau^T \hf(\hz) \\
	&= \langle \ell(\hf(\hz);\hz),\fs \rangle - \tau^T(\hf(\hz) - \fs) \\
	&\le L(\fs ;\hz,c) + \gamma(\mathcal{C}_k) L(\hf(\hz);\hz,c) \; ,
\end{align*}
where
\begin{align*}
	&\gamma(\mathcal{C}_k) = \sup_{\substack{c \in \mathcal{C}_k \\ \hz\in\mathbb{R}^{nm}_{\ge 0} \\ f \in \mathcal{F}}}\frac{\langle \ell(\hf(\hz);\hz,c) - \ell(f;\hz,c), f \rangle - \tau^T(\hf - f)}{\langle \ell(\hf(\hz);\hz,c), \hf \rangle} \\
	& \le \max_{\substack{c_i \in \mathcal{C}_k \\ \hz_i\in\mathbb{R}^{m}_{\ge 0} \\ f_i \in \mathbb{R}_{\ge 0}  }}\frac{ (\ell_i(\hf_i(\hz_i);\hz_i,c_i) - \ell_i(f_i;\hz_i,c_i)) f_i  - \tau_i(\hf_i - f_i)}{\ell_i(\hf_i(\hz_i),\hz_i,c_i) \hf_i } \\
	&\le \max_{\substack{a_i \in \mathcal{A}_k \\ \hz_i\in\mathbb{R}^{m}_{\ge 0} \\ f_i \in \mathbb{R}_{\ge 0} }}\frac{ f_i \big(\sum\limits_{j}^m a_i^j \hz^j_i - a^{j'}_i(f_i - \hspace{-10pt} \sum\limits_{j \in {[}j'-1{]}} \hspace{-6pt} \hz^j_i )  - a^m_i ( \hspace{-4pt} \sum\limits_{j \in {[}m{]}} \hspace{-4pt} \hz^j_i - f_i)  \big) } {( \sum_{j \in {[}m{]}} a_i^j \hz^j_i)(\sum_{j \in {[}m{]}} \hz^j_i) } \\
	&= \max_{\substack{a_i \in \mathcal{A}_k \\ \hz_i\in\mathbb{R}^{m}_{\ge 0} \\ f_i \in \mathbb{R}_{\ge 0}}}\frac{ f_i \big(\sum\limits_{j}^m (a^j_i-a^m_i ) \hz^j_i - (a^{j'}_i - a^m_i)(f_i - \hspace{-6pt} \sum\limits_{j \in {[}j'-1{]}} \hspace{-4pt} \hz^j_i )  \big) } {( \sum_{j \in {[}m{]}} a_i^j \hz^j_i)(\sum_{j \in {[}m{]}} \hz^j_i) } \\
	&\le \max_{\substack{a_i \in \mathcal{A}_k \\ \hz_i\in\mathbb{R}^{m}_{\ge 0} \\ f_i \in \mathbb{R}_{\ge 0}}}\frac{ f_i (a^{j'}_i-a^m_i )\big(\sum_{j}^m  \hz^j_i - (f_i -\sum_{j \in {[}j'-1{]}} \hz^j_i )  \big) } {a^{j'}_i( \sum_{j \in {[}m{]}}  \hz^j_i)(\sum_{j \in {[}m{]}} \hz^j_i) } \\
	&\le \max_{ \hz_i\in\mathbb{R}^{m}_{\ge 0} , f_i \in \mathbb{R}_{\ge 0}}\frac{ (k-1 )f_i \big(\sum_{j}^m  \hz^j_i - (f_i -\sum_{j \in {[}j'-1{]}} \hz^j_i )  \big) } {k( \sum_{j \in {[}m{]}}  \hz^j_i)(\sum_{j \in {[}m{]}} \hz^j_i) } \\
	&= \max_{ \hz_i\in\mathbb{R}^{m}_{\ge 0} , f_i \in \mathbb{R}_{\ge 0}}\frac{ (k-1 )f_i \big(\sum_{j \in {[}m{]}}  \hz^j_i - f_i   \big) } {k( \sum_{j \in {[}m{]}}  \hz^j_i)(\sum_{j \in {[}m{]}} \hz^j_i) }
\end{align*}
where the first inequality stems from the terms in the denominator being nonnegative. As before, without loss of generality we choose the vehicle indices such that on road $i$, $a^1_i \ge a^2_i \ge \hdots \ge a^m_i$. The second inequality stems from the constant term in the latency, $b_i$, being nonnegative and appearing only in the denominator. The third inequality stems from the fact that $a^j_i$ is monotonically nonincreasing in the index $j$, and the final inequality is from the definition of the maximum degree of inequality, $k$.  We can then maximize with respect to $f_i$ by setting $f_i = 1/2 \sum_{j \in {[}m{]}}  \hz^j_i$. Then,
\begin{align*}
	&\gamma(\mathcal{C}_k) \le \frac{ (k-1 )} {4k} \; \text{, so} \\
	&L(\hf(\hz);\hz,c) \le 4k/(3k-1) L(\fs ;\hz,c) \; .
\end{align*}
Then by Lemma~\ref{lma:agg_opt_cost}, $L(\fs ;\hz,c) \le k C(\sz)$, where $\sz$ is an optimal routing with respect to the latency function $c$. Together this proves the theorem. \hfill \qed

\subsection{Proof of Lemma~\ref{lma:same_roads_unique}}
The proof of this Lemma derives from the proof of Theorem~2 in \cite{lazar2020optimal}. The lemma statement yields the same conditions as the Theorem statement for the aforementioned theorem -- where the Theorem relies on toll $P$ to keep vehicle types on the roads in which they were in $\sz$, the lemma statement explicitly stipulates this condition. Then, the proof of the theorem in \cite{lazar2020optimal} provides statement 1. To prove the second statement, we note that the tolls are the path tolls of \cite{dafermos1973toll}, which, for a flow demand equal to that of the demand used to generate $\prevflow$, the flow for determining the tolls, guarantees the existence of an equilibrium with flow equal to $\prevflow$. This existence, with the uniqueness proven in the previous statement, proves the second lemma statement. \hfill \qed

\subsection{Proof of Lemma~\ref{lma:relative_costs}}
The reasoning for this statement is as follows. Due to the fact that $j$ was not on $i$ in $z^*$, and from the structure of the tolls, we know that $c^{\tau_j}_{i}(\hat{z}_{i}) = c^{\overline{\tau}}_i(\hat{z}_{i}) + \epsilon$. By definition of Wardrop Inequality, due to $\hat{z}^j_i > 0$, we know $c^{\tau_j}_i(\hat{z}_i) \le c^{\tau_j}_{i'}(\hat{z}_{i'})$ for all other roads $i'$ in ${[}n{]}$. Then, there is some road $i'$ such that $ c^{\tau_j}_{i'}(\hat{z}_{i'}) = c^{\overline{\tau}}_{i'}(\hat{z}_{i'})$, (since there is at least one road $i'$ with positive flow of type $j$ in $z^*$). From Wardrop Equilibrium we find $c^{\tau_j}_i(\hat{z}_i) \le c^{\tau_j}_{i'}(\hat{z}_{i'})$. As we've stated, $c^{\tau_j}_{i}(\hat{z}_{i}) = c^{\overline{\tau}}_i(\hat{z}_{i}) + \epsilon$, and since $i' \in \mathcal{N}^{z^*}_j$, we have $c^{\tau_j}_{i'}(\hat{z}_{i'}) = c^{\overline{\tau}}_{i'}(\hat{z}_{i'})$. This gives us
\begin{equation}
c^{\overline{\tau}}_{i}(\hat{z}_{i}) + \epsilon = c^{\overline{\tau}}_{i'}(\hat{z}_{i'}) \; .
\end{equation}

From Proposition \ref{prop:max_gap} and the definitions of $\underline{c}$ and $\overline{c}$ in \eqref{eq:min_max_cost}, we can state that
\begin{equation*}%\label{eq:standard_cost_min_max}
\begin{aligned}
	&c^{\overline{\tau}}_{i}(\hat{z}_{i}) = \underline{c} \quad \text{and} \\
	&c^{\overline{\tau}}_{i'}(\hat{z}_{i'}) = \overline{c} \; .
\end{aligned}
\end{equation*}
\vspace{-10pt} \hfill \qed

\subsection{Proof of Lemma~\ref{lma:std_cost_greater_flow}}
By definition of $\mathcal{I}_1$,
\begin{equation}
	\sum_{i \in \mathcal{I}_1} \hat{z}^j_i = \sum_{i \in \mathcal{I}_1} z^{*j}_i \quad \forall i' \in \mathcal{I}_1, \;  \forall j \in \mathcal{M}^{z^*}_{i'} \; .
\end{equation}

In words, all vehicle types on $\mathcal{I}_1$ in $z^*$ will be on those roads as well in $\hat{z}$, in addition to some vehicle flow that was not there in $z^*$. Then, we can show by contradiction that the resulting standard cost $c^{\overline{\tau}}_i$ for the roads in $\mathcal{I}_1$ will be greater than or equal to $\mu$ as follows. From \eqref{eq:standard_cost_partitions}, we know that the standard cost will be the same on all roads in $\mathcal{I}_1$. By definition of the toll structure, the standard cost of the roads in $\mathcal{I}_1$ with routing $z^*$ will be $\mu$. In $\hat{z}$ there is not less flow of any type to route on $\mathcal{I}_1$, and cost functions are nondecreasing in the vehicle flows. If $c^{\overline{\tau}}_i(\hat{z}_i) = c_i(\hat{z}_i) + \mu - c_i(z^{*}_i) < \mu$ for all roads in $\mathcal{I}_1$, then $z^*$ would not be an optimal routing, since this new routing can decrease the social cost from what it was in $z^*$. By this contradiction, 
\begin{equation*}%\label{eq:std_cost_partition_1}
	c^{\overline{\tau}}_i(\hat{z}_i) \ge \mu \quad \forall i \in \mathcal{I}_1 \; .
\end{equation*}
\vspace{-10pt} \hfill \qed

\subsection{Proof of Lemma~\ref{lma:std_cost_lesser_flow}}
By definition of $\mathcal{I}_2$,
\begin{equation}
	\sum_{i \in \mathcal{I}_2} \hat{z}^j_i \le \sum_{i \in \mathcal{I}_2} z^{*j}_i \quad \forall j \in {[}m{]} \; ,
\end{equation}
and each road has only the vehicle flow types that it did in $z^*$, since $\mathcal{I}_1 \cap \mathcal{I}_2 = \emptyset$ (from their different values of $c^{\overline{\tau}}_i$ of their roads -- otherwise we arrive at a contradiction earlier on). From Lemma~\ref{lma:same_roads_unique}, each road having the vehicle flow types that it did in $z^*$ implies the uniqueness of equilibrium flow. Since the cost functions are nondecreasing in the vehicle flows, the equilibrium flow $\hz$ has standard cost less than or equal to that of the routing in $z^*$, which is $\mu$. Accordingly, 
\begin{equation*}%\label{eq:std_cost_partition_2}
	c^{\overline{\tau}}_i(\hat{z}_i) \le \mu \quad \forall i \in \mathcal{I}_2 \; .
\end{equation*}
\hfill \qed

\end{document}